\documentclass[11pt, a4paper]{article}
\usepackage{amsfonts, amsmath, amssymb, amsthm, longtable, url, enumitem, mathtools, tikz}

\usepackage{xfrac}

\usepackage{caption, subcaption}
\usetikzlibrary{arrows,decorations.markings}
\renewcommand{\arraystretch}{1.5}

\theoremstyle{plain}
\newtheorem{theorem}{Theorem}[section]

\newtheorem{corollary}[theorem]{Corollary}
\newtheorem{lemma}[theorem]{Lemma}
\newtheorem{proposition}[theorem]{Proposition}
\newtheorem*{claim*}{Claim}
\newtheorem*{problem*}{Problem}
\newtheorem*{conjecture*}{Conjecture}

\theoremstyle{definition}
\newtheorem{definition}[theorem]{Definition}

\newcommand\al{\alpha}
\newcommand\bt{\beta}

\newcommand\dl{\delta}
\newcommand\lm{\lambda}
\newcommand{\sg}{\sigma}
\newcommand\cF{\mathcal{F}}

\newcommand\cM{\mathcal{M}}

\newcommand\la{\langle}
\newcommand\ra{\rangle}
\newcommand\lla{\langle\!\langle}
\newcommand\rra{\rangle\!\rangle}

\newcommand\ad{\mathrm{ad}}
\newcommand\End{\mathrm{End}}

\newcommand{\ch}{\mathrm{char}}
\newcommand{\gr}{\mathrm{gr}}

\newcommand{\sgn}{\mathrm{sgn}}

\newcommand\N{\mathbb{N}}
\newcommand\Z{\mathbb{Z}}
\newcommand\FF{\mathbb{F}}

\newcommand\Aut{\mathrm{Aut}}
\newcommand{\A}{\mathrm{A}}
\newcommand{\B}{\mathrm{B}}
\newcommand{\C}{\mathrm{C}}

\newcommand{\Y}{\mathrm{Y}}

\newcommand{\Miy}{\mathrm{Miy}}
\newcommand{\Cl}{\mathrm{Cl}}
\newcommand{\Mab}{\mathcal M(\alpha, \beta)}
\newcommand{\IY}{\mathrm{IY}}

\newcommand\cH{\mathcal{H}}
\newcommand{\hatH}{\hat{\cH}}
\newcommand\cL{\mathcal{L}}
\newcommand{\hatL}{\hat{\cL}}

\renewcommand{\phi}{\varphi}
\renewcommand{\epsilon}{\varepsilon}

\newcommand{\0}{\overline{0}}
\newcommand{\1}{\overline{1}}
\newcommand{\2}{\overline{2}}
\newcommand{\ii}{\overline{\imath}}
\newcommand{\jj}{\overline{\jmath}}

\renewcommand{\r}{\overline{r}}
\renewcommand{\t}{\overline{t}}
\renewcommand{\a}{\overline{a}}
\renewcommand{\b}{\overline{b}}
\renewcommand{\c}{\overline{c}}
\newcommand{\s}{\overline{s}}
\newcommand{\kk}{\overline{k}}

\newcommand{\qa}{\bar{a}}
\newcommand{\qs}{\bar{s}}

\newcommand{\pattern}{pattern}

\newcommand{\ta}{{\tau_{\sfrac{3}{2}}}}
\newcommand{\flip}{{\tau_{\sfrac{1}{2}}}}

\newcommand{\tw}{\tilde{w}}

\setlist[enumerate,1]{label={\upshape\arabic*.}}
\setlist[enumerate,2]{label={\upshape (\alph*)}}
\setlist[enumerate,3]{label={\upshape (\roman*)}}

\usepackage{ltablex} 
\usepackage{tabularx} 
\keepXColumns

\newcolumntype{C}[1]{>{\centering\arraybackslash}m{#1}}
\newcolumntype{Y}{>{\centering\arraybackslash}X}

\title{Quotients of the Highwater algebra and its cover}

\author{C.~Franchi\footnote{Dipartimento di Matematica e Fisica,
Universit\`a Cattolica del Sacro Cuore, Via Garzetta 48, I-25133 Brescia, Italy, email: clara.franchi@unicatt.it}
 \and
 M.~Mainardis\footnote{Dipartimento di Scienze Matematiche, Informatiche e Fisiche, 
Universit\`a degli Studi di Udine, via delle Scienze 206, I-33100 Udine, Italy, email: mario.mainardis@uniud.it}
 \and
 J.~M\textsuperscript{c}Inroy\footnote{School of Mathematics, University of Bristol, Fry Building, Woodland Road, Bristol, BS8 1UG, UK, and the Heilbronn Institute for Mathematical Research, Bristol, UK, email: justin.mcinroy@bristol.ac.uk}
 }

\date{}

\begin{document}
\maketitle

\begin{abstract}
Axial algebras are a class of non-associative algebra with a strong link to finite (especially simple) groups which have recently received much attention.  Of primary interest are the axial algebras of Monster type $(\alpha, \beta)$, of which the Griess algebra (with the Monster as its automorphism group) is an important motivating example.  In this paper, we complete the classification of the symmetric $2$-generated primitive axial algebras of Monster type $(\alpha, \beta)$.

By previous work of Yabe \cite{yabe}, and Franchi and Mainardis \cite{highwater5}, any such algebra is either explicitly known, or is a quotient of the infinite-dimensional Highwater algebra $\cH$, or its characteristic $5$ cover $\hatH$.  In this paper, we classify the ideals of $\cH$ and $\hatH$ and thus their quotients.  Moreover, we give explicit bases for the ideals.  In fact, we proceed in a unified way, by defining a cover $\hatH$ of $\cH$ in all characteristics and classifying its ideals.  Our new algebra $\hatH$ has a previously unseen fusion law and provides an insight into why the Highwater algebra has a cover which is of Monster type only in characteristic $5$.
\end{abstract}

\section{Introduction}

Recently several collections of finite simple groups, such as $3$-transposition groups and many of the sporadic groups including the Monster, have been realised as an automorphism group of an axial algebra of Monster type $(\al, \bt)$.  In fact, the originating example of an axial algebra is the Griess algebra, which has the Monster sporadic simple group as its automorphism group.  The Griess algebra turns out to be an axial algebra of Monster type $(\frac{1}{4}, \frac{1}{32})$.

As for many algebraic structures (Lie algebras being one notable example), a crucial step towards an understanding of the structure of these algebras is the classification of the $2$-generated objects.  The primitive $2$-generated axial algebras of Monster type $(\frac{1}{4}, \frac{1}{32})$ were first classified by Pasechnik, Ivanov, Seress, and Shpectorov in~\cite{IPSS10}, extending earlier work of Norton~\cite{Norton} and Sakuma~\cite{Sakuma}, and are known as Norton-Sakuma algebras.  It was Rehren in~\cite{RT, R} who first introduced the generalisation to Monster type $(\al, \bt)$ and began a systematic study of the primitive $2$-generated algebras that are \emph{symmetric} -- those that admit an automorphism switching the two generators.  He generalised the eight Norton-Sakuma algebras to eight families of examples.  Vijay Joshi introduced some new families of Monster type $(2\bt, \bt)$ in \cite{JoshiPhD, doubleMatsuo}.  In an unexpected development, Franchi, Mainardis, and Shpectorov in \cite{highwater}, and independently Yabe in \cite{yabe}, found the infinite-dimensional $2$-generated Highwater algebra $\cH$, which is of Monster type $(2, \frac{1}{2})$.

A major breakthrough came from Yabe, who gave in~\cite{yabe} an almost complete classification of the symmetric $2$-generated primitive axial algebras of Monster type in characteristic other than $5$. The remaining case was considered by Franchi and Mainardis in~\cite{highwater5}, who introduced a characteristic $5$ cover $\hatH$ of the Highwater algebra and showed that all the cases not included in Yabe's classification are factors of $\hatH$.  Putting these all together we have the following:

\begin{theorem}\textup{\cite{yabe, highwater5}}
\label{existing}
A symmetric $2$-generated primitive axial algebra of Monster type $(\al,\bt)$ is isomorphic to one of the following:
\begin{enumerate}
\item a $2$-generated primitive axial algebra of Jordan type $\al$, or $\bt$;
\item a quotient of $\cH$, or $\hatH$ in characteristic $5$;
\item one of the algebras in a family listed in \textup{\cite[Table $2$]{yabe}}.
\end{enumerate}
\end{theorem}

The $2$-generated primitive axial algebras of Jordan type, were classified by Hall, Rehren and Shpectorov in~\cite{Axial2} and are of dimension at most $3$.  Every algebra in case $(3)$ above is known and of dimension at most $8$.  In this paper, we complete the classification by classifying all the quotients of the Highwater algebra $\cH$ and of its characteristic $5$ cover $\hatH$.  Moreover, we give explicit bases for the ideals and so also the quotients.

Before discussing our results, we briefly recall the definition of axial algebras.  An \emph{axial algebra} is a non-associative algebra $A$ generated by a set $X$ of \emph{axes}.  These axes are primitive semisimple idempotents, whose eigenvectors multiply according to a so-called \emph{fusion law} $\cF$.  When the fusion law is $C_2$-graded (which the Monster fusion law $\cM(\al, \bt)$ is), then we have an algebra automorphism $\tau_a$ called a \emph{Miyamoto involution}, associated to each axis $a \in X$.  The group generated by all these automorphisms is called the \emph{Miyamoto group}.

In order to give a unified proof of our main result for the Highwater algebra and its characteristic $5$ cover, we introduce a new algebra $\hatH$ which is a cover of the Highwater algebra in all characteristics\footnote{Note that the Highwater algebra does not exist in characteristics $2$ and, in characteristic $3$, it is not of Monster type. So we exclude these throughout the paper.}.  By Theorem~\ref{existing}, in characteristics other than $5$, this cannot be an axial algebra of Monster type $(\al, \bt)$.  In fact, we show that $\hatH$  has a new fusion law $\cF$ with entries $\{ 1,0,\frac{5}{2}, 2, \frac{1}{2} \}$ (note that in characteristic $5$, $\frac{5}{2} = 0$).  This algebra has a basis given by axes $a_i$, $i \in \Z$, and additional elements $s_j$ and $p_{\r,j}$, for $j \in \N$ and $r = 1,2$.  For the multiplication, see the full definition in Definition~\ref{seconddefn}.

\begin{theorem}\label{HW5}
The algebra $\hatH$ is a symmetric $2$-generated primitive axial algebra with fusion law given in Table~$\ref{Htable}$.
\end{theorem} 

\begin{table}[htb!]
\begin{center}
\renewcommand{\arraystretch}{1.5}
\begin{tabular}[t]{c||c|c|c|c|c}
& $1$ & $\frac{5}{2}$ & $0$ & $2$ & $\frac{1}{2}$ \\
\hline
\hline
$1$ & $1$ & $\frac{5}{2}$ &  & $2$ & $\frac{1}{2}$ \\
\hline
$\frac{5}{2}$ & $\frac{5}{2}$ & $\frac{5}{2}$ & $\frac{5}{2}$ & & $\frac{1}{2}$ \\
\hline
$0$ & & $\frac{5}{2}$ & $\frac{5}{2}, 0$ & $\frac{5}{2}, 2$ & $\frac{1}{2}$\\
\hline
$2$ & $2$ &  & $\frac{5}{2}, 2$ & $\frac{5}{2}, 0$ & $\frac{1}{2}$\\
\hline
$\frac{1}{2}$  & $\frac{1}{2}$ & $\frac{1}{2}$ & $\frac{1}{2}$ &  $\frac{1}{2}$ & $\frac{5}{2}, 0, 2$
\end{tabular}
\caption{The fusion law $\cF$ for $\hatH$}\label{Htable}
\end{center}
\end{table}

In Proposition \ref{aut}, we show that the full automorphism group $\Aut(\hatH)$ of $\hatH$ is isomorphic to $D_\infty$ and it acts naturally on the indices of the set of axes $X = \{ a_i : i \in \Z \}$.  In particular, there is an automorphism $\flip$ which switches the two generating axes $a_0$ and $a_1$ and so $\hatH$ is indeed symmetric. The Miyamoto group $\Miy(\hatH) = \la \tau_i : i \in \Z \ra \cong D_\infty$ has index $2$ in $\Aut(\hatH)$.

The algebra $\hatH$ has a distinguished ideal $J$ given by the subspace $\la p_{\r,j} : j \in \N, r = 1,2 \ra$.

\begin{theorem}\label{1.3}
$J$ is an ideal of $\hatH$ and the quotient $\hatH/J$ is isomorphic to the Highwater algebra $\cH$.  In characteristic $5$, $\hatH$ coincides with the characteristic $5$ cover of the Highwater algebra defined in~\textup{\cite{highwater5}}.
\end{theorem}

So our algebra $\hatH$ is indeed a cover of the Highwater algebra as claimed.  Thus classifying the ideals, whence the quotients, of $\hatH$ will simultaneously classify the quotients of the Highwater algebra and its characteristic $5$ cover.  By \cite{axialstructure}, every ideal in an axial algebra is invariant under the action of the Miyamoto group.  In fact, we show a stronger result for $\hatH$ which will prove crucial in classifying the ideals.

\begin{theorem}
All ideals of $\hatH$ are $\Aut(\hatH)$-invariant. In particular, every quotient of $\hatH$ and so also every quotient of $\cH$ is symmetric.
\end{theorem}

Using a sort of Euclidean division algorithm on $\hatH$, we show the following.

\begin{theorem}
Every ideal of $\hatH$ is principal.
\end{theorem}

This has the following important consequence which highlights the distinguished nature of the ideal $J$.  Note that, by definition, $J$ has infinite codimension.

\begin{corollary}
An ideal $I \unlhd \hatH$ has finite codimension if and only if it is not contained in the ideal $J$.  Every non-trivial ideal which is contained in $J$ has finite codimension in $J$.
\end{corollary}

Based on this, we may split our ideals into two classes, those which are contained in $J$ and those which are not.

Consider first an ideal $I$ which is not contained in $J$; this has finite codimension in $\hatH$.  In particular, the images of the axes in $\hatH/I$ span a subspace of some finite dimension $D$ and so we say that $I$ has \emph{axial codimension} $D$.  Using the fact that $I$ is $\Aut(\hatH)$-invariant, we show that there exists $x = \sum_{i=0}^D \al_i a_i \in I$, for some $\al_i \in \FF$ where $\al_0 \neq 0 \neq \al_D$.  Moreover, $x$ generates an ideal of axial codimension $D$. Since any ideal containing $x$ can be recovered by finding the corresponding ideal of the finite dimensional quotient $\hatH/(x)$, we restrict ourselves to classifying the minimal ideals of axial codimension $D$. 

For such an ideal $I = (x)$, we want to find conditions on the tuple of elements $(\al_0, \dots, \al_D)$.  We begin by showing in Proposition \ref{baric}, that $\hatH$ is a baric algebra.  That is, there is an algebra homomorphism $\lm \colon \hatH \to \FF$ which is given by $\lm(a_i) = 1$ and $\lm(s_j) = 0 = \lm(p_{\r,j})$.  This immediately gives a (Frobenius) symmetric bilinear form $( \cdot, \cdot)\colon \hatH \times \hatH \to \FF$ defined by $(y,z) = \lm(y)\lm(z)$.  Using standard results from \cite{axialstructure}, we show that any proper ideal $I$ of $\hatH$ lies in the radical $\hatH^\perp = \ker(\lm)$ of the form $(\cdot, \cdot)$.  Thus if
\[
\sum \al_i a_i + \sum \bt_j s_j + \sum \gamma_{\r,k}p_{\r,k}
\]
is an element of  $I$, then $\sum \al_i = 0$. In particular this has to hold for the coefficients of $x$. This is our first key observation.  For our second key observation, we use a minimality argument and the $\Aut(\hatH)$-invariance of $I$ to see that there exists $\epsilon = \pm 1$ such that $\al_{i} = \epsilon \al_{D-i}$, for all $i = 0, \dots, D$. We say that $x$ and the tuple $(\al_0, \dots, \al_D)$ are of \emph{$\epsilon$-type}.

In fact, these two key observations are the only two restrictions on the generator $x =\sum_{i=0}^D \al_i a_i$ of a such and ideal $I$.  We define a tuple $(\al_0, \dots, \al_D) \in \FF^{D+1}$ to be of \emph{ideal-type} if $\alpha_0 \neq 0 \neq \alpha_D$, $\sum_{i=0}^D \al_i = 0$, and $(\alpha_0, \ldots , \alpha_D)$ is of $\epsilon$-type, for $\epsilon = \pm 1$.  

\begin{theorem}
For every $D\in \N$, there is a bijection between the set of ideal-type $(D+1)$-tuples $(\alpha_0, \ldots , \alpha_D)\in \FF^{D+1}$, up to scalars, and the set of minimal ideals of axial codimension $D$ of $\hatH$ given by
\[
(\alpha_0, \ldots , \alpha_D) \mapsto \left(\sum_{i=0}^D \al_i a_i\right).
\]
\end{theorem}

Moreover, in Theorem \ref{idealspan}, we give an explicit basis for each such ideal $I$ and hence all the maximal quotients of $\hatH$.

Since the Highwater algebra $\cH$ is isomorphic to $\hatH/J$, no non-trivial ideal of the Highwater algebra corresponds to an ideal contained in $J$ and so the above results explicitly describe all the quotients of $\cH$ with maximal axial codimension.  To complete the classification of the symmetric $2$-generated primitive algebras of Monster type, we now turn to classifying the ideals of our algebra $\hatH$ which are contained in $J$.  Recall that $J = \la p_{\r,j} : j \in \N, r = 1,2 \ra$.

\begin{theorem}
There is a bijection between the set of tuples ${(\bt_3, \dots,\bt_{3k}) \in \FF^k}$, for $k\in \N$, up to scalars, and the ideals $I \subseteq J$ of $\hatH$, given by
\[
 (\bt_3, \dots,\bt_{3k})\mapsto \left(\sum_{j=1}^k \bt_{3j}p_{\1,3j} \right).
\]
\end{theorem}

Again, in Theorem \ref{Jidealbasis}, we give an explicit basis for each such ideal $I$ contained in $J$.

We end the paper by describing all isomorphisms between quotients of $\hatH$ (and hence also of the Highwater algebra $\cH$) and other symmetric $2$-generated algebras of Monster type (those in cases (1) and (3) of Theorem \ref{existing}).

\begin{theorem}
The only isomorphisms between quotients of $\hatH$ and $\cM(2, \frac{1}{2})$-axial algebras in cases $(1)$ and $(3)$ of Theorem $\ref{existing}$ are with $3\C(2)$, $S(2)^\circ$, $\widehat{S}(2)^\circ$, $\IY_3(2, \frac{1}{2}, \mu)$, for $\mu \in \FF$, $\IY_5(2, \frac{1}{2})$ and $6\A(2, \frac{1}{2})$ in characteristic $5$, \textup{(}and their quotients\textup{)}.
\end{theorem}

(See Section \ref{sec:yabe} for the explicit ideals for each isomorphism.)  All the above isomorphisms are with algebras $A$ which are $\cM(2, \frac{1}{2})$-axial algebras.  The other possibility is if the quotient of $\hatH$ also has a grading with respect to the eigenvalue $2 \in \cF$ and if $A$ is a $\cM(\frac{1}{2},2)$-axial algebra.  In Theorem~\ref{2graded exceptional isos}, we determine all such examples and find that the only possibilities for such isomorphisms are with $3\C(2)$, $6\Y(\frac{1}{2}, 2)$, or $\IY_3(2, \frac{1}{2}, 1)$ and the quotient must be a quotient of the Highwater algebra $\cH$.

The paper is organised as follows.  In Section \ref{sec:background}, we give a brief overview of axial algebras.  Our main actor $\hatH$ is introduced in Section \ref{sec:algebra}, where we show that it is a cover of the Highwater algebra and determine its automorphism group.  In Section \ref{sec:fusionlaw}, we prove that $\hatH$ has the fusion law given in Table \ref{Htable} and hence is an axial algebra.  We give its Frobenius form and some preliminary results on ideals in the brief Section \ref{sec:idealprelim}.  We show in Section \ref{sec:idealsauto} that ideals of $\hatH$ are $\Aut(\hatH)$-invariant.  Ideals contained in $J$ are classified in Section~\ref{sec:infinite}, showing also that they are principal.  Principality of the remaining ideals is shown in Section \ref{sec:principal}.  Ideals which are not contained in $J$ are classified in Section \ref{sec:ideals not in J}, where we also give explicit bases for them.  Two important families of examples are given in Section \ref{sec:twofamilies} which allows us to show our exceptional isomorphisms in Section \ref{sec:yabe}.

%%%%%%%%%%%%%%%%%%%%%%%%%%%%%%%%
\section{Background}
\label{sec:background}

For an algebra $A$ over a field $\FF$ and $X\subseteq A$, we'll denote by $\langle X \rangle$ the linear span of the set $X$ and by $\lla X\rra$ the subalgebra generated by $X$.

For an element $a \in A$, denote by $A_\lambda (a) = \{ v : av = \lambda v\}$ the $\lambda$-eigenspace for the adjoint $\ad_a$.  For ease of notation, for ${\mathcal N}\subseteq \FF$, define
\[
A_{\mathcal N}(a):=\bigoplus_{\lambda \in {\mathcal N}}A_\lambda (a).
\]

A \emph{fusion law} is a pair $\cF = (\cF, \star)$ where $\cF$ is a non-empty set and $\star \colon \cF \times \cF \to 2^{\cF}$ is a symmetric map.  It will be convenient to extend the map $\star$ to subsets of $\cF$ in the obvious way.

Given a non-associative algebra $A$ over $\FF$ and a fusion law $\cF$, an \emph{$\cF$-axis} (or simply an \emph{axis} when there no ambiguity in the choice of $\cF$) is an idempotent element $a$ of $A$ such that
\begin{enumerate}
\item[(Ax1)] $ad_a \colon v\mapsto av$ is a semisimple endomorphism of $A$ with spectrum contained in 
$\cF$;
\item[(Ax2)] for every $\lm,\mu\in \cF$,  
 \[
A_\lm(a) A_\mu(a) \subseteq A_{\lm \star \mu}(a) = \bigoplus_{\nu \in \lm\star\mu} A_\nu
\]
\end{enumerate} 
Furthermore, $a$ is called {\it primitive} if 

\begin{enumerate}
\item[(Ax3)] the $1$-eigenspace of $ad_a$ is $\langle a \rangle$.
\end{enumerate}

An \emph{axial algebra} over $\FF$ with a fusion law $\cF$ is a commutative non-associative $\FF$-algebra $A$ generated by a set $X$ of $\cF$-axes. If all the elements of $X$ are primitive, $A$ is called \emph{primitive}.

For an abelian group $T$, a \emph{$T$-grading} of the fusion law $\cF$ is a map 
$\gr\colon \cF \to T$ such that, for every $\lambda, \mu$ in $\cF$
\[
\gr(\lm \star \mu) \subseteq \{\gr(\lambda)\gr(\mu)\}.
\]
 A $T$-grading of $\cF$ is a \emph{finest grading} if every other grading of $\cF$ factors through the grading $T$.  By \cite[Proposition 3.2]{D}, every fusion law admits a unique finest grading.  A $T$-grading $\gr$ is \emph{adequate} if the image $\gr(\cF)$ generates $T$.  So we may always assume that our grading is adequate.
We are most interested in the case where $T = \Z_2$.  Taking $\Z_2 = \{ +, - \}$, for a $\Z_2$-grading $\gr$, denote by $\cF_+$ and $\cF_-$ the full preimages via $\gr$ of $+$ and $-$ respectively.  For every axis $a$ of $A$, a grading on the fusion law induces a grading on the algebra: for $\epsilon \in \{+,-\}$, we set 
\[
A_\epsilon(a):=A_{\cF_\epsilon}(a).
\]

A straightforward computation shows that the map that negates $A_-(a)$ and induces the identity on $A_+(a)$ is an involutory algebra automorphism called \emph{Miyamoto involution} associated to the axis $a$ (see~\cite{Miya02, IPSS10}).
The group generated by all the Miyamoto involutions associated to the axes in $X$ is called the \emph{Miyamoto group} $\Miy(X)$ (see~\cite{axialstructure}).  Note that the Miyamoto group is not always the full automorphism group of the algebra, as is the case for the algebra $\hatH$ considered in this paper.

An axial algebra $A$ is \emph{$2$-generated} if there are two axes $a$ and $b$ in $A$ such that $A= \lla a, b \rra$. Further, we say that $A$ is \emph{symmetric} if there exists an involutory automorphism $f$ of $A$ (which could be in the Miyamoto group) that switches the two generating axes $a$ and $b$.

The fusion law $\Mab$ in Table~\ref{Ising} is called the \emph{Monster fusion law} (in the table, we omit the set symbols for the entries $\lambda\star \mu$ and so in particular empty entries correspond to the empty set).  Manifestly, this fusion law is $\Z_2$-graded, with $\cF_+ = \{ 1,0, \al\}$ and $\cF_- = \{ \bt\}$.  We say an axial algebra $A$ is of \emph{Monster type $(\al,\bt)$} if it has the Monster fusion law $\Mab$.

\begin{table}[htb!]
\[ 
\begin{array}{c||c|c|c|c}
 & 1 & 0 & \al & \bt\\
\hline
\hline
1 & 1 & & \al & \bt\\
\hline
0 & & 0 & \al & \bt\\
\hline
\al	& \al & \al & 1,0 & \bt\\
\hline
\bt & \bt & \bt & \bt & 1,0,\al\\
\end{array}
\]
\caption{Fusion law $\Mab$}\label{Ising}
\end{table} 

A {\it Frobenius form} on an $\FF$-algebra $A$ is a non-zero symmetric bilinear form  $$
\kappa\:\colon A\times A \to \FF$$ that {\it associates} with every element of $A$, that is, for every $x,y,z$ in $A$,  
$$\kappa(x,yz)=\kappa(xy,z).$$
 From the above formula it follows immediately that the radical of a Frobenius form on $A$ is a (two-sided) ideal.

%%%%%%%%%%%%%%%%%%%%%%%%%%%%%%%%

\section{The algebra $\hatH$}
\label{sec:algebra}

In this section, we will define the main actor in this paper, the algebra $\hatH$ which will be a cover of the Highwater algebra in all characteristics.  We first give a definition in the style of the characteristic $5$ cover of the Highwater algebra as given in \cite{highwater5} and we then introduce a second definition with respect to a different basis.  This new basis will be more useful for us throughout the rest of the paper, simplifying many arguments.  We will show that in characteristic other than $3$, the two definitions are equivalent.  (Since the Highwater algebra is not $2$-generated in characteristic $3$, this will not matter for our goal.)

Throughout the paper, we adopt the following notation.  Let $\FF$ be a field of characteristic not $2$.   For $r\in \Z$, we denote by $\overline{r} \in \Z_3$ the congruence class $r +3\Z$.

Define $\dl \colon \Z_3 \to \FF$ by $\dl(\0) = 0$, $\dl(\1) = 1$ and $\dl(\2) = -1$.

\begin{definition}\label{firstdefn}
Let $\hatH$ be an algebra over $\FF$ with basis $\{ a_i : i \in \Z \} \cup \{ s_{\0, j} :  j \in \N \} \cup \{ s_{\1, 3j}, s_{\2, 3j} :  j \in \N \}$.  We set $s_{\0,0} = 0$ and $s_{\overline{r},j} = s_{\0,j}$ when $j \notin 3\N$.  Define multiplication on $\hatH$ by 
\begin{enumerate}
\item $a_i a_j := \frac{1}{2}(a_i+a_j) +s_{\overline{\imath}, |i-j|}$
\item $a_i  s_{\overline{r}, j} := -\frac{3}{4} a_i + \frac{3}{8}( a_{i-j}+ a_{i+j}) +\frac{3}{2} s_{\overline{r} ,j} + \dl(\overline{\imath} - \overline{r})( s_{\overline{r} -\1,j}- s_{\overline{r}  + \1,j})$
\item $s_{\overline{r}, j}  s_{\overline{t}, k}:= \frac{3}{4}( s_{\overline{r}, j}+ s_{\overline{t}, k}) - \frac{1}{8} \sum_{x = 0,1,2} (s_{\bar{x}, |j-k|} + s_{\bar{x}, j+k})$, if $\{i,j\}\not \subseteq 3\N$
\item $s_{\a, 3j}  s_{\b, 3k} := \frac{3}{4}\sum_{h = j,k} (s_{\a, 3h}+ s_{\b, 3h}- s_{-(\a+\b), 3h}) - \frac{3}{8} \sum_{h = |j-k|, j+k} (s_{\a, 3h}+ s_{\b, 3h}- s_{-(\a+\b), 3h})$
\end{enumerate}
\end{definition}

Note that if $\{ \a, \b, \c\} = \{ \0, \1, \2\}$, then $-(\a+\b) = \c$, but if $\b = \a$, then $-(\a+\b) = \a$.

It is clear that if $\ch(\FF) = 5$, then $\hatH$ is precisely the cover of the Highwater algebra as defined by Franchi and Mainardis in \cite{highwater5}.  They showed that there is an ideal $J = \la s_{\0,j} - s_{\2,j}, s_{\1,j} - s_{\0,j} : j \in 3\N \ra$ and $\hatH/J \cong \cH$.  So in characteristic $5$, the differences of the $s_{\r,j}$ span $J$ and thus play a fundamental role in $\hatH$.  We mirror this by defining some new elements of $\hatH$ in any characteristic other than $3$.
\begin{align*}
s_j &:= \tfrac{1}{3} \sum_{\r \in \Z_3} s_{\r,j} \\
p_{\r,j} &:= \tfrac{1}{3}\left( s_{\r-\1,j} - s_{\r+\1,j} \right) \\
z_{\r,j} &:= p_{\r+\1,j} - p_{\r-\1,j}
\end{align*}
Note that if $j \notin 3\N$, then $s_j = s_{\0,j}$, $p_{\r,j} = 0$ and so $z_{\r,j}=0$.  Also $\sum_{\r \in \Z_3} p_{\r,j} = 0$ and so $\sum_{\r \in \Z_3} z_{\r,j} = 0$ also.

We now give our second definition which has a more natural basis.

\begin{definition}\label{seconddefn}
Let $\hatH$ be an algebra over $\FF$ with basis 
\[
{\mathcal B}:=\{ a_i : i \in \Z \} \cup \{ s_{j} :  j \in \N \}  \cup \{p_{\r,k}: \r\in \{\overline{1},\overline{2}\} \mbox{ and } k\in 3\N\}
\]
We set $s_0 = 0$, $p_{\r,j} = 0$ for all $\r \in \Z_3$ if $j \notin 3\N$, $p_{\0,j} := -p_{\1,j} -p_{\2,j}$ and $z_{\r,j} = p_{\r+\1,j} - p_{\r-\1,j}$.  Define multiplication on $\hatH$ by
\begin{enumerate}
\item[H1] $a_i a_j := \frac{1}{2}(a_i+a_j) +s_{|i-j|}+z_{\ii,|i-j|}$
\item[H2] $a_i  s_{j} := -\frac{3}{4} a_i + \frac{3}{8}( a_{i-j}+ a_{i+j}) +\frac{3}{2} s_{j} -z_{\ii,j}$
\item[H3] $a_i p_{\r,j}:=\frac{3}{2}p_{\r,j} - p_{-(\ii+\r),j} $
\item[H4] $s_j  s_ l:= \frac{3}{4}( s_{j}+ s_{l}) - \frac{3}{8}(s_{|j-l|} + s_{j+l})$
\item[H5] $s_j p_{\r,k}:=\frac{3}{4}( p_{\r, j}+ p_{\r, k}) - \frac{3}{8}(p_{\r, |j-k|} + p_{\r, j+k})$
\item[H6] $
 p_{\r,h} p_{\t,k}:= \frac{1}{4}(z_{-(\r+\t), h}+ z_{-(\r+\t),  k})
- \frac{1}{8}(z_{-(\r+\t), |h-k|}+z_{-(\r+\t), h+k})
 $
\end{enumerate}
where $i \in \Z$, $j,l \in \N$, $h,k \in 3\N$ and $\r,\t \in \Z_3$.
\end{definition}

Note that, as $p_{\r-\1,j} + p_{\r,j} + p_{\r+\1,j} = 0$, we have $3p_{\r,j} = z_{\r-\1,j} - z_{\r+1,j}$.  We now immediately justify our use of the same letter $\hatH$ for both algebras.

\begin{lemma}
Suppose that $\ch(\FF) \neq 3$, then the algebras in Definitions~$\ref{seconddefn}$ and $\ref{firstdefn}$ are isomorphic.  The isomorphism is given by
\[
a_i \mapsto a_i, \quad s_j \mapsto \tfrac{1}{3}\sum_{\r \in \Z_3} s_{\r,j}, \quad  p_{\r,j} \mapsto \tfrac{1}{3}(s_{\r-\1,j} - s_{\r+\1,j}).
\]
Note that the inverse maps $a_i \mapsto a_i$ and $s_{\r,j} \mapsto s_j + z_{\r,j}$.
\end{lemma}
\begin{proof}
This is immediate from checking the multiplication.
\end{proof}

Where we do not rule out characteristic $3$ going forward, we will use the second definition.

For all characteristics $\hatH$ is a cover of the Highwater algebra, extending the definition in \cite{highwater5}.

\begin{lemma}\label{Jideal}
Let $J$ be the subspace $\la p_{\1,j} , p_{\2,j}  : j \in \N \ra$ of $\hatH$.  Then $J$ is an ideal of $\hatH$ and the quotient $\hatH/J$ is isomorphic to the Highwater algebra $\cH$.
\end{lemma}
\begin{proof}
From H3, H5, and H6 it is clear that $AJ \subseteq J$ and so $J$ is an ideal.  It is now straightforward to see that $\hatH/J \cong \cH$ from the definition of $\cH$.
\end{proof}

We begin by determining the automorphism group of $\hatH$.  For $k \in \frac{1}{2}\Z$, let $\tau_k \colon\Z \to \Z$ be the reflection in $k$ given by $i \mapsto 2k-i$.  Then $D:=\langle \tau_0, \flip \rangle$ is the infinite dihedral group acting naturally on $\Z$.  Let $\sgn \colon D \to \Z$ be the sign representation of $D$.  That is, $\sgn(\rho) = -1$ if $\rho$ is a reflection and $\sgn(\rho) = 1$ if $\rho$ is a translation.

For $\rho \in D$, define $\phi_\rho \colon \hatH \to \hatH$ to be the linear map given by
\[
{a_i}^{\phi_\rho} = a_{i^\rho}, \quad {s_j}^{\phi_\rho} = s_j, \quad {p_{\r,k}}^{\phi_\rho}=(\sgn \:\rho) p_{\overline{r^\rho},k}
\]
Note that, we have $z_{\r,k}^{\phi_\rho}=z_{\overline{r^\rho},k}$.

\begin{proposition}\label{Zaction}
For every $\rho\in D$, $\phi_\rho$ is an automorphism of ${\hatH}$ and the map $\rho\mapsto \phi_\rho$ defines a faithful representation of $D$ as a subgroup of automorphisms of $\hatH$.
\end{proposition}

\begin{proof}
By the above formulas, the products H2-H6 are preserved by $\phi_\rho$.  For H1, observe that the action of $D$ on $\Z$ preserves distance, whence $|i-j| = |i^\rho-j^\rho|$ for all $i,j \in \Z$.  It is then clear that $\phi$ is an injective group homomorphism since every $\phi_\rho$ acts non-trivially on $\{a_i : i\in \Z\}$ if $\rho \neq 1$.
\end{proof}

By an abuse of notation, from now on we identify $\phi_\rho$ with $\rho$. To determine the automorphism group of $\hatH$ we use the following fact.

\begin{lemma}\label{idems}
If $\ch(\FF) \neq 3$, then the only non-trivial idempotents in $\hatH$ are the $a_i$'s, $i \in \Z$.
\end{lemma}
\begin{proof}
The proof is analogous to that of~\cite[Lemma 2.3]{highwater}.
\end{proof}

\begin{proposition}\label{aut}
If $\ch(\FF) \neq 3$, then $\Aut(\hatH) \cong D$.
\end{proposition}
\begin{proof}
The proof is a modified version of \cite[Proposition 2.4]{highwater}, which we sketch here. Since $\ch(\FF) \neq 3$, we may use Definition~\ref{firstdefn}. By Lemma~\ref{idems}, $\Aut(\hatH)$ permutes the non-trivial idempotents $\{ a_i : i \in \Z \}$ and so it permutes the set of $s_{\r,j} = a_ia_j - \frac{1}{2}(a_i+a_j)$.  For the pair $(\r, k)$, where $\r \in \Z_3$ and $k \in \N$, we define a graph $\Gamma_{(\r,k)}$ with vertices $\{a_i : i \in \Z\}$ and an edge between $a_i$ and $a_j$ if and only if $s_{\r, k}=a_ia_j-\frac{1}{2}(a_i+a_j)$. It is easy to see that if $k \notin 3\Z$, then $\Gamma_{(\r,k)}$ has exactly $k$ connected components, while if $k \in 3\Z$,  $\Gamma_{(\r,k)}$ has $\frac{k}{3}$ connected components. Since two graphs $\Gamma_{(\r,k)}$ and $\Gamma_{(\bar t,l)}$ are isomorphic if and only if $k=l\in 3\Z$, it follows that $\Aut(\hatH)$ fixes every $s_{\r,k}$, if $k\not \in 3\Z$ and permutes the set $\{s_{\0,k}, s_{\1, k}, s_{\2, k}\}$ when $k\in 3\Z$.
\end{proof}

It will turn out that the Miyamoto automorphism $\tau_{a_i}$ associated to an axis $a_i$ coincides with the reflection $\tau_i$ and the automorphism swapping $a_0$ with $a_1$ (usually denoted by $f$ in the context of symmetric axial algebras of Monster type) is $\tau_\frac{1}{2}$.  We let $\theta_j$, for $j\in \Z$, be the automorphism of $\hatH$ induced by the translation on $\Z$ by $j$.

We record in the next lemma the products with $z_{\r,j}$ which will be useful later.

\begin{lemma}\label{prodz}
For $i\in \Z$, $j \in \N$, $h,k \in 3\N$ and $\{ \r, \t\} \subseteq \Z_3$, we have the following.
\begin{enumerate}
\item $a_i z_{\r,j}=\frac{3}{2}z_{\r,j} + z_{-(\ii+\r),j}$
\item $s_j z_{\r,k}=\frac{3}{4}( z_{\r,j}+ z_{\r, k}) - \frac{3}{8}(z_{\r, |j-k|} + z_{\r, j+k})$
\item $p_{\r,h}z_{\t,k}=\frac{3}{4}( p_{-(\r+\t),h}+ p_{-(\r+\t),k}) - \frac{3}{8}(p_{-(\r+\t), |h-k|} + p_{-(\r+\t), h+k})$
\item $z_{\r,h}z_{\t,k}= -\frac{3}{4}( z_{-(\r+\t),h}+ z_{-(\r+\t),k}) + \frac{3}{8}(z_{-(\r+\t), |h-k|} + z_{-(\r+\t), h+k})$
\end{enumerate}
\end{lemma}
\begin{proof}
These follow immediately from the multiplication in $\hatH$.
\end{proof}

We close this section with the following observation.
Define $\lm \colon \hatH \to \FF$ by $\lm(a_i) = 1$ and  $\lm(s_{j}) = 0=\lm(p_{\r,k})$ and extend linearly.

\begin{proposition}\label{baric}
The map $\lm$ is an algebra homomorphism and so $\hatH$ is a baric algebra.
\end{proposition}
\begin{proof}
The proof is analogous to that of \cite[Lemma 2.2]{highwater}.
\end{proof}

As an immediate consequence, the map $(\cdot, \cdot) \colon \hatH\times \hatH \to \FF$, defined by $(x,y) = \lm(x)\lm(y)$, is a Frobenius form.

%%%%%%%%%%%%%%%%%%%%%%%%%%%%%%%%%

\section{The fusion law}\label{sec:fusionlaw}

Let $\cF$ be the fusion law on the set $\{1,0,\frac{5}{2}, 2, \frac{1}{2}\} \subseteq \FF$ described in Table~\ref{Htable} on page \pageref{Htable}. In this section we prove Theorem~\ref{HW5}.

\begingroup
\def\thetheorem{\ref{HW5}}
\begin{theorem}
If $\ch(\FF) \neq 2,3$, then $\hatH$ is a symmetric $2$-generated primitive axial algebra with fusion law $\cF$ given in Table $\ref{Htable}$.
\end{theorem}
\addtocounter{theorem}{-1}
\endgroup

We begin by showing that $a_0$ is an $\cF$-axis. By the action of $\Aut(\hatH)$, this implies that, for all $i\in \Z$, $a_i$ is an $\cF$-axis. Similarly to \cite{highwater} and \cite{highwater5}, we consider an `$i$-slice' of the algebra.  Let
\[
U^i := \la a_0, a_{-i}, a_i, s_{i}, p_{\1,i}, p_{\2,i} \ra
\]
Note that, if $i \notin 3\N$, then $p_{\r,i} = 0$ and $U^i$ is $4$-dimensional.  Otherwise, if $i \in 3\N$, then $U^i$ is $6$-dimensional.
Since $\ad_{a_0}$ preserves $U^i$, for all $i \in \N$, and $\hatH = \bigoplus_{i \in \N} U^i$, we restrict to $U^i$ to find the eigenvectors of $\ad_{a_0}$.  For $i \notin 3\N$, the analysis is the same as in \cite{highwater} and $\ad_{a_0}$ is semisimple with eigenvalues $1$, $0$, $2$ and $\frac{1}{2}$.  So let $i \in 3\N$.  Then the action of $\ad_{a_0}$ on $U^i$, with respect to the basis $(a_0, a_{-i}, a_i, s_{i}, p_{\1,i}, p_{\2,i})$, is represented by the following matrix
\[
\begin{pmatrix} 1 & 0 &0&0&0&0 \\ 
\frac{1}{2}&\frac{1}{2}&0&1&1&-1 \\ 
\frac{1}{2}&0&\frac{1}{2}&1&1&-1 \\ 
-\frac{3}{4}&\frac{3}{8}&\frac{3}{8}&\frac{3}{2}&-1&1 \\ 
0 &0 &0  & 0&\frac{3}{2}&-1 \\ 
0 &0 &0  & 0 &-1&\frac{3}{2}
 \end{pmatrix}
\]
which has eigenspaces $U^i_1 = \la a_0 \ra$, $U^i_0 = \la u_i \ra$, $U^i_2 = \la v_i \ra$, $U^i_\frac{5}{2} = \la z_{i} \ra$ and $U^i_\frac{1}{2} = \la w_i, \tw_i \ra$, where
\begin{align*}
u_i &:= 6a_0 - 3(a_{-i} + a_i) + 4s_{i}+4z_{\0,i}\\
v_i &:= 2a_0 - (a_{-i} + a_i) - 4s_{i}-4z_{\0,i}\\
w_i &:= a_{-i} - a_i\\
z_{i} &:= p_{\1,i} - p_{\2,i} =z_{\0,i}\\
\tw_{i} &:= p_{\1,i} +p_{\2,i} = -p_{\0,i}
\end{align*}
Notice that, in any characteristic other than $2$, if $i \notin 3\N$, $(a_0, u_i, v_i, w_i, z_i, \tw_i)$ is a basis for $U^i$.
On the other hand, when $i \in 3\N$, since $p_{\1,i} = p_{\2,i}=0$, $z_i$ and $\tw_i$ are zero and $(a_0, u_i, v_i, w_i)$ is a basis for $U^i$. 
We set $\hatH_1:=\langle a_0\rangle$, $\hatH_u:=\langle u_i \:\mid \:i \in \N\rangle$,  $\hatH_z:=\langle z_i \:\mid \:i \in \N\rangle$, $\hatH_v:=\langle v_i \:\mid \:i \in \N\rangle$ and $\hatH_w:=\langle w_{i}, \tilde w_i \:\mid \:i \in \N\rangle$.
Moreover, define $\hatH^+:=\hatH_1\oplus \hatH_u\oplus \hatH_z\oplus \hatH_v$  and $\hatH^-:= \hatH_w$.
Finally, it will be convenient to set $u_0:=v_0:=w_0:=z_0:=\tw_0 :=0$. We immediately have the following.

\begin{lemma}\label{deca}
With the above notation, $\hatH=\hatH_1\oplus \hatH_u\oplus \hatH_z\oplus \hatH_v\oplus \hatH_w=\hatH^+\oplus \hatH^-$.
\end{lemma}

\begin{lemma}\label{grading}
The involution $\tau_{0}$ acts as the identity on $\hatH^+$ and as minus identity on $\hatH^-$. In particular, $\hatH^+$ is the fixed subalgebra of $\tau_0$, $\hatH^+\hatH^-\subseteq  \hatH^-$ and $\hatH^-\hatH^-\subseteq  \hatH^+$. 
\end{lemma}
\begin{proof}
It follows immediately from the definition of $\tau_0$, since $\tau_0$ is an algebra automorphism. 
\end{proof}

Recall from Proposition \ref{baric}, that $\hatH$ is baric with respect to the algebra homomorphism $\lm \colon \hatH \to \FF$ defined by $\lm(a_i) = 1$ and  $\lm(s_{j}) = 0=\lm(p_{\r,k})$.  We can use this to refine the condition that $\hatH^-\hatH^-\subseteq  \hatH^+$.

\begin{lemma}\label{neg baric}
We have $\hatH^-\hatH^-\subseteq  \hatH_u\oplus \hatH_z\oplus \hatH_v$.
\end{lemma}
\begin{proof}
It is immediate to see that, for every $\eta\in \{u,v,z, w\}$ and $x\in \hatH_\eta$, $\lambda(x)=0$, whereas $\lm(\hatH_1) = 1$.  Since $\lm$ is an algebra homomorphism and $\hatH^-\hatH^-\subseteq  \hatH^+$, the result follows immediately.
\end{proof}

To calculate the fusion law, it will be convenient to use different elements, which generalise those in \cite{highwater5}.
We set  $c_0:=0$ and, for $i \in \N$, we define
\[
c_i := 2a_0 - (a_{-i}+a_i)
\]
which allows us to rewrite $u_i$ and $v_i$ as
\begin{align*}
u_i &= 3c_i+4 s_{ i}+4z_i \\
v_i &= c_i- 4s_{ i}-4z_i.
\end{align*}
In order to calculate the products of such elements, we also introduce the following.  For $i,j\in \N$ we define 
\begin{align*}
c_{i,j} &:= -2(c_i+c_j) +c_{|i-j|} + c_{i+j}\\
t_{i,j} &:= -2(s_{i} + s_{j}) + s_{ |i-j|} + s_{ i+j}\\
u_{i,j}&:=-2( u_i + u_j)+ u_{|i-j|}+ u_{i+j}\\
v_{i,j}&:=-2( v_i+ v_j)+ v_{|i-j|}+ v_{i+j}\\
z_{i,j}&:=-2(z_{i}+z_{j})+z_{|i-j|}+z_{i+j}.
\end{align*}

Firstly, note that all the above expressions are symmetric in $i$ and $j$.  Secondly, $u_{i,j} \in \hatH_u$, $v_{i,j} \in \hatH_v$ and $z_{i,j} \in \hatH_z$.  To calculate the products of our basis vectors, we begin by computing the products with our new elements.

\begin{lemma} \label{transition}
For all $i,j\in\N$, we have 
\begin{enumerate}
\item $ c_i c_j= \begin{cases}
2t_{i,j}+2z_{i,j}&\mbox{if } i \in 3\N\\
2t_{i,j}+2z_{i,j}-3(z_{|i-j|}+z_{i+j})& \mbox{otherwise}
\end{cases} $
\item 
$c_i s_{j} = \begin{cases}
  \mathrlap{\frac{3}{8}c_{i,j}} \hphantom{2t_{i,j}+2z_{i,j}-3(z_{|i-j|}+z_{i+j})}&\mbox{if } i \in 3\N\\
  \frac{3}{8}c_{i,j} -3z_{j}& \mbox{otherwise}
  \end{cases}$
\item $c_iz_{j}=\begin{cases}
\mathrlap{0} \hphantom{2t_{i,j}+2z_{i,j}-3(z_{|i-j|}+z_{i+j})}&\mbox{if } i \in 3\N\\
3z_{j}& \mbox{otherwise}
\end{cases}$
\end{enumerate}
\end{lemma}

\begin{proof}
A straightforward computation gives:
\begin{align*}
 c_i c_j&=(2a_0-( a_{-i}+ a_i))(2 a_0-( a_{-j}+ a_j)) \\
&=4 a_0-2\left [\tfrac{1}{2}(a_{-i}+ a_0)+ s_{i}+z_{\0,i}+\tfrac{1}{2}(a_i+ a_0)+ s_{i}+z_{\0,i}\right ] \\
	&\phantom{{}= 4a_0{}} -2\left [\tfrac{1}{2}(a_0+a_{-j})+ s_{j}+z_{\0,j}+\tfrac{1}{2}( a_0+ a_j)+ s_{j}+z_{\0,j}\right ]\\
	&\phantom{{}=4a_0{}} +\left [\tfrac{1}{2}(a_{-i}+ a_{-j})+ s_{|i-j|}+z_{-\ii,|i-j|} +\tfrac{1}{2}(a_i+a_{-j})+ s_{i+j}+z_{\ii,i+j}\right .\\
	&\phantom{ {}=4a_0 + \big[ {} } \left.+\tfrac{1}{2}(a_{-i}+a_j)+ s_{i+j}+z_{-\ii,i+j}+\tfrac{1}{2}(a_i+a_j)+ s_{|i-j|}+z_{\ii,|i-j|}\right ]\\
&= -4(s_{i}+ s_{j})+ 2(s_{|i-j|}+ s_{i+j})-4(z_{\0,i}+z_{\0,j})\\
	&\phantom{{}= 4a_0{}}+z_{\ii,|i-j|}+z_{-\ii,|i-j|}+z_{\ii,i+j}+z_{-\ii,i+j}\\
&=2t_{i,j}+ 2z_{i,j}-2(z_{\0,|i-j|}+z_{\0,i+j})+z_{\ii,|i-j|}+z_{-\ii,|i-j|}+z_{\ii,i+j}+z_{-\ii,i+j}.
\end{align*}
Since $z_{\0,j}+z_{\ii,j}+z_{-\ii,j}=0$ if $i\not \in 3\N$, the result follows.
For the second assertion we have
\begin{align*}
 c_i s_{j}&= [2 a_0-( a_{-i}+ a_i)] s_{j}\\
&=2\left [-\tfrac{3}{4} a_0+\tfrac{3}{8}( a_{-j}+ a_j)+\tfrac{3}{2}s_{ j}- z_{\0,j}\right]\\
&\phantom{{}={}} -\left [-\tfrac{3}{4}  a_{-i}+\tfrac{3}{8}( a_{-i-j}+ a_{-i+j})+\tfrac{3}{2}s_{j}-z_{-\ii,j}\right ]\\
&\phantom{{}={}} -\left [-\tfrac{3}{4}  a_{i}+\tfrac{3}{8}( a_{i-j}+ a_{i+j})+\tfrac{3}{2} s_{j}- z_{\ii,j}\right ] \\
&= -\tfrac{3}{4}c_i - \tfrac{3}{4}c_j + \tfrac{3}{8}c_{|i-j|} + \tfrac{3}{8}c_{i+j}-2z_{\0,j}+z_{\ii,j}+z_{-\ii,j}.
\end{align*}
As above, the result follows.
For the third assertion we have
\begin{align*}
 c_i z_{j}&= [2 a_0-( a_{-i}+ a_i)] z_{\0,j}=\\
&=2\left (\tfrac{3}{2} z_{\0,j}+z_{\0,j}\right) -\left (\tfrac{3}{2}  z_{\0,j}+z_{\ii,j}\right ) -\left (\tfrac{3}{2}  z_{\0,j}+z_{-\ii,j}\right )\\
&= 2z_{\0,j}-z_{\ii,j}-z_{-\ii,j}
\end{align*}
and the result follows as in the previous case.
\end{proof}

We can also express the products of the $s_{j}$'s and $z_j$'s in a compact form using our new elements.

\begin{lemma}\label{ss}
For all $i,j\in\N$, $h,k\in 3\N$ we have 
\begin{enumerate}
\item $\mathrlap{s_{ i}s_{ j }} \hphantom{z_hz_k} = -\frac{3}{8}t_{i,j}$
\item $z_hz_k = \frac{3}{8}z_{h,k}$
\item $\mathrlap{s_iz_h} \hphantom{z_hz_k}= -\frac{3}{8}z_{i,h}$.
\end{enumerate}
\end{lemma}
\begin{proof}
This is immediate from the definition of $t_{i,j}$ and H4 and the definition of $z_{i,j}$ and Lemma~\ref{prodz}.
\end{proof}

We now use our new elements to rewrite $u_{i,j}$ and $v_{i,j}$.

\begin{lemma}\label{uij}
For all $i,j\in\N$, we have 
\begin{enumerate}
\item $ u_{i,j} =3c_{i,j}+4t_{i,j}+4z_{i,j}$,
\item $ \mathrlap{v_{i,j}} \hphantom{u_{i,j}}= c_{i,j}-4t_{i,j} -4z_{i,j}.$
\end{enumerate}
\end{lemma}
\begin{proof}
By the definition of our basis elements we have
\begin{align*}
u_{i,j}&=-2(3c_{i}+4s_{i}+4z_i)-2(3c_{j}+4s_{j}+4z_j)\\
&\phantom{{}={}}+(3c_{|i-j|}+4s_{|i-j|}+4z_{|i-j|}) +(3c_{i+j}+4s_{i+j}+4z_{i+j})\\
&= 3c_{i,j}+4t_{i,j}+4z_{i,j}.
\end{align*}
and similarly for $v_{i,j}$.
\end{proof}

We may now determine the fusion law, by finding the products between the elements which span the different parts.

\begin{lemma} \label{product uu}
For all $i,j\in\N$, we have
\begin{enumerate}
\item
$
u_i u_j=\begin{cases}
\mathrlap{3u_{i,j}} \hphantom{-3v_{i,j} -15z_{i,j}} & \mbox{if } ij\in 3\N\\
3u_{i,j}-21z_{i,j}& \mbox{otherwise} 
\end{cases}
$
\item 
$
\mathrlap{u_i v_j} \hphantom{u_iu_j} = \begin{cases}
-3v_{i,j} & \mbox{if } ij\in 3\N\\
-3v_{i,j} -15 z_{i,j} & \mbox{otherwise} 
\end{cases}
$
\item 
$
\mathrlap{v_i v_j} \hphantom{u_iu_j} =\begin{cases}
\mathrlap{-u_{i,j}} \hphantom{-3v_{i,j} -15 z_{i,j}}& \mbox{if } ij\in 3\N\\
-u_{i,j} +3 z_{i,j}& \mbox{otherwise} 
\end{cases}
$
\item 
$
\mathrlap{u_iz_j} \hphantom{u_iu_j} =\begin{cases}
\mathrlap{0} \hphantom{-3v_{i,j} -15 z_{i,j}} & \mbox{if } i\in 3\N\\
12z_j & \mbox{otherwise}
\end{cases}
$
\item 
$
\mathrlap{v_iz_j} \hphantom{u_iu_j} =0
$
\end{enumerate}
\end{lemma}
\begin{proof}
By the definition of $u_i$ and Lemmas~\ref{transition} and \ref{ss}, we have
\begin{align*}
 u_i u_j&=( 3c_i+4 s_{ i}+4z_{i})( 3c_j+4 s_{ j}+4z_{j})\\
&=9 c_i c_j+12\left(c_i(s_{j}+ z_{j})+ c_j(s_{i} + z_{i})\right)\\
&\phantom{{}={}}+ 16(s_{i} s_{ j}+s_i z_{j} +s_j z_{i} +z_{i}z_{j})\\
&= 9c_{i,j} -6 t_{i,j} + 9c_ic_j + 16(s_i z_{j} +s_j z_{i} +z_{i}z_{j})
\end{align*}
If $i,j \notin 3\N$, then $z_i = 0 = z_j$ and $(z_{|i-j|}+z_{i+j})=z_{i,j}$.  By Lemma \ref{transition}, $c_{i,j} = 2t_{i,j}+2z_{i,j} - 3(z_{|i-j|} + z_{i+j}) = 2t_{i,j} - z_{i,j}$ and hence $u_i u_j = 9c_{i,j} - 6t_{i,j} +18t_{i,j} -9z_{i,j} = 3u_{i,j} -21z_{i,j}$.   Now suppose that $i\in 3\N$ and $j\not \in 3\N$ (by symmetry the case where $i\not \in 3\N$ and $j\in 3\N$ is equal to this).  By Lemma~\ref{transition} and Lemma~\ref{ss}, we get $u_i u_j = 9c_{i,j} -6 t_{i,j} + 18t_{i,j}+18z_{i,j} -6z_{i,j}=3u_{i,j}$.  Finally, if $i,j \in 3\N$, we get $u_i u_j = 9c_{i,j}-6t_{i,j}+18t_{i,j}+18z_{i,j}-6z_{i,j}-6z_{i,j}+6z_{i,j}=3u_{i,j}$.

For the second and third formulas we have
\begin{align*}
 u_i v_j&=( 3c_i+4 s_{ i}+4z_{i})( c_j-4 s_{j}-4z_{j})\\
&= 3c_i c_j -12 {c_i(s_{j}+ z_{j})+4c_j( s_{i} + z_{i})} -16 s_{ i} s_{j}-16s_iz_{j}\\
	&\phantom{{}={}}-16s_jz_{i}-16z_{i}z_{j}\\
&=  6t_{i,j} -3c_{i,j} + 3c_i c_j-16 (s_iz_{j} + s_jz_{i} + z_{i}z_{j})
\end{align*}
and 
\begin{align*}
 v_i v_j&=( c_i-4 s_{i}-4z_{i})( c_j-4 s_{j}-4z_{j})\\
&= c_i c_j-4\left( {c_i(s_{j}+z_i)+  c_j(s_{i} + z_{i})}\right)\\
&\phantom{{}={}}+ 16(s_{i} s_{ j}+s_i z_{j} +s_j z_{i} +z_{i}z_{j})\\
&= -3c_{i,j} {-6t_{i,j}} + c_i c_j+ 16(s_i z_{j} +s_j z_{i} +z_{i}z_{j})
\end{align*}
and the result follows from Lemma~\ref{transition} and Lemma~\ref{ss} as in the previous case. The last two claims follow in a similar way. 
\end{proof}

We may now prove the main result of this section.

{

\begin{proof}[Proof of Theorem~$\ref{HW5}$]
It is clear from the multiplication that $a_0$ is an idempotent and, by Lemma~\ref{deca}, $\ad_{a_0}$ is semisimple with eigenvalues $1$, $\frac{5}{2}$, $0$, $2$ and $\frac{1}{2}$.  The fusion law follows from Lemmas~\ref{neg baric}, \ref{product uu} and \ref{prodz} and so $a_0$ is an $\cF$-axis.  By Proposition~\ref{Zaction}, using the action of $D_\infty$ (which is $\Aut(\hatH)$ if $\ch(\FF) \neq 3$), $a_i$ is an $\cF$-axis, for all $i \in \Z$.  Therefore $\hatH$ is a primitive axial algebra with the fusion law given in Table~\ref{Htable}.

Observe that, for every $i\in \Z$, the map $\tau_i$ is precisely the Miyamoto involution associated to the axis $a_i$. 
The fact that $\hatH$ is $2$-generated and symmetric follows in a similar way to \cite[Theorem 2.1]{highwater} and \cite[Theorem 6]{highwater5}, but is more involved. 
Since $\tau_0$ maps $a_1$ to $a_{-1}$, $\lla a_0, a_1\rra = \lla a_0, a_1, a_{-1} \rra$.  Now $\lla a_0, a_1, a_{-1} \rra$ is invariant under $\tau_{\frac{1}{2}}$, whence $\lla a_0, a_1 \rra$ is invariant under $\Aut(\hatH) = \la \tau_0, \tau_\frac{1}{2} \ra$.  However, $\Aut(\hatH)$ acts transitively on the $a_i$, so $a_i \in \lla a_0, a_1 \rra$, for all $i \in \Z$.  
By H1, $s_j + z_{\r,j} \in \lla a_0, a_1 \rra$, for all $j \in \N$, $\r\in \Z_3$.  As $\sum_{\r \in \Z_3} z_{\r,j} = 0$ and $\ch( \FF) \neq 3$, we may take linear combinations to get $s_j, z_{\r,j} \in \lla a_0, a_1 \rra$, for all $j \in \N$, $\r\in \Z_3$.  Finally, taking linear combinations of $z_{\0,j} = p_{\1,j} - p_{\2,j}$ and $z_{\1,j} = p_{\2,j} - p_{\0,j} = p_{\1,j} + 2p_{\2,j}$, we get $p_{\1,j}, p_{\2,j} \in \lla a_0, a_1 \rra$ for all $j \in \N$.
\end{proof}

Finally, we consider the case where $\ch(\FF) = 3$; then $\tfrac{5}{2}=1$ and $2=\frac{1}{2}$. However, the five part decomposition and the multiplication between those parts remains true, so $\hatH$, with respect to the set of axes $\{a_i\mid i\in \Z\}$, is an axial decomposition algebra as defined in~\cite{D}.  One can show that it is also a (non-primitive) axial algebra with the fusion law in Table~\ref{table:ch3}.

\begin{table}
$$ 
\begin{array}{|c||c|c|c|c|c|}
\hline
\star & 1 & 0 & \frac{1}{2}\\
\hline
\hline
1 & 1 & &\frac{1}{2}\\
\hline
0 & & & 1,\frac{1}{2}\\
\hline
\frac{1}{2}  & \frac{1}{2} & 1,\frac{1}{2} &0,1, \frac{1}{2}\\
\hline
\end{array}
$$
\caption{The fusion law for $\hatH$ when $\ch(\FF)=3$}\label{table:ch3}
\end{table} 
 Unlike the Highwater algebra, the idempotents $a_i$ do not satisfy the Jordan type fusion law, so $\hatH$ is not a Jordan algebra (see~\cite[p.~33]{Schafer}). 
 
Moreover, every pair of axes $a_i, a_j$ generates a $3$-dimensional Jordan algebra $\la a_i, a_j, s_{|i-j|} + z_{\ii,|i-j|} \ra$, which is isomorphic to the algebra $\widehat{S}(2)^\circ$ using the notation from \cite{forbidden} (this was called $\mathrm{Cl}^{00}(\FF, b_2)$ in~\cite[Theorem~(1.1)]{Axial2}).  In particular, $\hatH$ is not $2$-generated any more -- it is not even finitely generated.

Hence for the remainder of the paper we will assume that $\ch(\FF) \neq 2,3$.

%%%%%%%%%%%%%%%%%%%%%%%%%%%%%%%
\section{The Frobenius form and the radical}\label{sec:idealprelim}

In this short section, we prove some preliminary results about ideals of $\hatH$.  Recall from the end of Section~\ref{sec:algebra} that $\hatH$ has a Frobenius form $(\cdot, \cdot)$ defined by $(x,y)=\lambda(x)\lambda(y)$. Hence we may apply some general results from~\cite{axialstructure} about ideals in axial algebras with a Frobenius form to $\hatH$.

We can split ideals into two classes: those which don't contain any axes and those which do.  The \emph{radical} is the unique largest ideal which doesn't contain any axes.

\begin{lemma}
The radical of $\hatH$ is a codimension $1$ ideal spanned by the set $\{ a_i-a_j, s_k, p_{\r,k} : i,j \in \Z, k \in \N, \r \in \Z_3 \}$.
\end{lemma}
\begin{proof}
Since $(\cdot, \cdot)$ is a Frobenius form on $\hatH$ that is non-zero on the axes, by \cite[Theorem 4.9]{axialstructure}, the radical equals the radical of the Frobenius form.
\end{proof}

Ideals which contain an axis are controlled by the \emph{projection graph}.

\begin{lemma}
The projection graph for $\hatH$ is connected.
\end{lemma}
\begin{proof}
We have $(a,b) = \lm(a)\lm(b) = 1$ for all axes $a$ and $b$.  So, by \cite[Lemma 4.17]{axialstructure}, the projection graph is undirected and connected.
\end{proof}

\begin{corollary}\label{ideals in radical}
Every proper ideal of $\hatH$ is contained in the radical, in particular no proper ideal of $\hatH$ contains any axes.
\end{corollary}

Since every proper ideal $I$ is contained in the radical, we have the following result which we will use frequently.

\begin{corollary}\label{coeffsum}
Let $I$ be an ideal of $\hatH$, $v \in I$. Then the sum of the coefficients of the $a_i$ elements of $v$ is $0$.
\end{corollary}

We finish this section by noting two results which will be important for us.  Firstly, by~\cite[Corollary~3.11]{axialstructure}, ideals of an axial algebra are invariant under the Miyamoto group, so we get the following.

\begin{lemma}\label{idealMiy}
Every ideal of $\hatH$ is $\tau_i$-invariant, for every $i\in \Z$.
\end{lemma}

Secondly, the decomposition of the algebra into eigenspaces induces a decomposition of an ideal $I$ into a sum of eigenspaces.

\begin{lemma}\label{idealeigen}
Let $I \unlhd A$ be an ideal of an $\mathcal F$-axial algebra $A$ and $a \in A$ be an axis.  Then $I = \bigoplus_{\lm \in \cF} I_\lm$, where $I_\lm = I_\lm(a) = I \cap A_\lm(a)$ for all $\lm \in \cF$.
\end{lemma}

%%%%%%%%%%%%%%%%%%%%%%%%%%%%%%%%%%%%%%%%%%%%%%%%%%%%%%%%%%%%%%%%%%%%%%%%%%%
\section{Invariance of ideals under automorphisms}\label{sec:idealsauto}

As we have seen, in an axial algebra, every ideal is invariant under the action of the Miyamoto group.  For $\hatH$, the Miyamoto group is an index $2$ subgroup of the full automorphism group.  In this section, we show that ideals of $\hatH$ are in fact invariant under the full automorphism group.

\begin{theorem}\label{symmetry}
If $\ch(\FF) \neq 2,3$, then all ideals of $\hatH$ are $\Aut(\hatH)$-invariant.
\end{theorem}

\begin{corollary}\label{corsymmetry}
If $\ch(\FF) \neq 2,3$, then every quotient of $\hatH$ is a symmetric $2$-generated axial algebra.
\end{corollary}
\begin{proof}
Let $I \unlhd \hatH$.  By Corollary \ref{ideals in radical}, no axis is contained in $J$.  So the quotient $\hatH/I$ is generated by the images $\qa_0 = a_0 + I$ and $\qa_1 = a_1 + I$.  Since $I^{\flip} \subseteq I$, $\flip$ induces an automorphism of $\hatH/I$ which switches the two generators.
\end{proof}

We will prove Theorem \ref{symmetry} via a series of lemmas using the following strategy.  Let $I$ be a proper ideal of $\hatH$.  We must show that $I^\sg \subseteq I$ for all $\sg \in \Aut(\hatH)$.  By Lemma~\ref{idealeigen}, $I = \bigoplus_{\lm \in \cF} I_\lm$, where $I_\lm := I \cap \hatH_\lm(a_0)$.  So, since an automorphism is a linear map, it suffices to show that ${I_\lm}^\sg \subseteq I$ for all $\lm \in \cF$.  In fact, by Corollary~\ref{ideals in radical}, no non-trivial ideal contains an axis, so $I_1 = 0$ and we only need to consider $\lm \in \cF \setminus \{1\}$.

Recall that $\la \Miy(X), \flip \ra = \Aut(\hatH)$.  Since ideals are invariant under the action of the Miyamoto group,  it is enough to show that ${I_\lm}^\sigma \subseteq I$, for some element $\sigma \in \Aut(\hatH)$ such that $\la \Miy(X), \sigma \ra = \Aut(\hatH)$.
For some values of $\lm$ the most convenient choice for $\sigma$ is $\flip$ itself.  However, for other values, it is more convenient to use $\ta = {\flip}^{\tau_1}$ instead.  (This is because, for $j \in \N$, $\ta$ fixes $z_{\0,j}$ and swaps $z_{\1,j}$ and $z_{\2,j}$.)

For each $\lm \in \cF \setminus \{ 1, \tfrac{1}{2}\}$, we will show that for all $x \in \hatH_\lm$,
\[
x^\sg = F_\lm(x)
\]
where $F_\lm$ is contained in the subalgebra of $\End_{\FF}(\hatH)$ generated by the adjoint maps and 
the
elements of the Miyamoto group. Since these map $I$, and hence $I_\lambda$, into $I$, the result follows.  When $\lambda=\tfrac{1}{2}$, the argument is similar except that we further split $\hatH_{\sfrac{1}{2}}$ into two direct summands and treat each summand separately.

Finally, since $F_\lm$ is linear, it suffices to show that $x^\sg =  F_\lm(x)$ holds for $x$ in a basis of $\hatH_\lm$.  We must pay special attention to the characteristics where any of $\frac{5}{2},0,2, \frac{1}{2}$ coincide.  In particular, since we already assumed $\ch(\FF) \neq 2,3$, the only possibility is in characteristic $5$, where $\tfrac{5}{2} = 0$, in which case the $0$-eigenspace is generated by the $u_i$'s and the $z_i$'s.

From the definitions we immediately have the following.

\begin{lemma}\label{ztw}
For every $i\in \N$ we have
\begin{enumerate}
\item 
${z_i}^{\ta}=z_i$
\item
$\tw_i^{\ta}=-\tw_i.$
\end{enumerate}
\end{lemma}

\begin{corollary}\label{cor5/2}
If $\ch(\FF) \neq 5$, ${I_{\sfrac{5}{2}}}^{\Aut(\hatH)} \subseteq  I$.
\end{corollary}
\begin{proof}
Since $\ch(\FF) \neq 5$, we have $\hatH_{\sfrac{5}{2}} = \hatH_z$ and the result follows from Lemma~\ref{ztw}(1).
\end{proof}

We now compute the action of $\ta$ on the $u_i$ and $v_i$ eigenvectors for $a_0$.

\begin{lemma}\label{preliminary}
For every $i\in \N$, 
\begin{enumerate}
\item ${u_i}^{\ta} = 6a_3 - 3(a_{3-i} + a_{3+i}) + 4s_i+4z_{\0,i}$,
\item ${v_i}^{\ta} = 2a_3 - (a_{3-i} + a_{3+i}) - 4s_i- 4z_{\0,i}$,
\item $c_{3,i}= \frac{1}{3}(-2u_i + {u_i}^{\ta} + {u_i}^{\ta\tau_0})=-2v_i + {v_i}^{\ta} + {v_i}^{\ta\tau_0}$.
\end{enumerate}
\end{lemma}
\begin{proof}
The first two formulas follow immediately from the definitions of $u_i, v_i$ and $\ta$.  For the last formula, assume first that $i >3$. Then $|3-i| = 3-i$ and we have
\begin{align*}
c_{3,i} &= -2(c_3 + c_i) + c_{i-3} + c_{i+3} \\
&= -2c_i -2(2a_0-(a_{-3} + a_3))\\
&\phantom{{}={}} + 2a_0 -(a_{3-i} + a_{i-3}) + 2a_0 -(a_{-3-i} + a_{i+3}) \\
&= -2c_i +2a_3 - (a_{3-i} + a_{3+i}) + 2a_{-3} -(a_{-3-i} + a_{-3+i}) \\
&= -2c_i + {c_i}^{\ta} + {c_i}^{\ta\tau_0}.
\end{align*}
A similar argument holds for $i =1,2,3$.  Since $s_i$ and $z_{\0,i}$ are invariant under $\ta$ and $\tau_0$ and $c_i=\tfrac{1}{3}(u_i-4s_i-4z_{\0,i})=v_i+4s_i+4z_{\0,i}$, we get the last claim.
\end{proof}

We may now write identities for ${u_i}^{\ta}$ and ${v_i}^{\ta}$, giving us the endomorphisms $F_0$ and $F_2$.

\begin{lemma}\label{uf}
For every $i\in \N$ we have
\begin{enumerate}
\item $
{u_i} ^{\ta}= u_i-\frac{5}{4}a_3u_i+\frac{3}{4}a_{-3}u_i+s_{3}u_i+z_3u_i,$
\item 
$
{v_i} ^{\ta}= \frac{7}{12}a_3v_i-\frac{1}{12}a_{-3}v_i+\frac{1}{3}s_{3}v_i+\frac{1}{3}z_{3}v_i.$
\end{enumerate}
\end{lemma}
\begin{proof}
For the first claim, multiply $u_i$ by $a_3$ and use Lemmas~\ref{prodz} and \ref{preliminary}, to get
\begin{align*}
a_3 u_i &= a_3\left(3c_i + 4s_i+4z_{\0,i}\right) \\
&= \tfrac{3}{2} c_i + 6s_3+6z_{\0,3} - 3(s_{|i-3|}+z_{\bar 3,|i-3|} + s_{i+3}+z_{\bar 3,i+3}) \\
&\phantom{{}={}} -3a_3 + \tfrac{3}{2}(a_{3-i} + a_{3+i}) + 6s_i+6z_{\0,i} \\
&= \tfrac{1}{2}(3c_i + 4s_i+4z_{\0,i}) - \tfrac{1}{2}(6a_3 -3(a_{3-i} + a_{3+i}) + 4s_i+4z_{\0,i}) \\
&\phantom{{}={}} + 6(s_3+ s_i) -3(s_{|i-3|} +s_{i+3}) +6(z_{\0,3} +z_{\0,i}) -3(z_{\0,|i-3|}+z_{\0,i+3})\\
&= \tfrac{1}{2}(u_i-{u_i}^{\ta})-3(t_{3,i}+z_{3,i}).
\end{align*}
Apply the map $\tau_0$ to the above equality (noting that ${t_{3,i}}^{\tau_0} = t_{3,i}$ and ${z_{3,i}}^{\tau_0} = z_{3,i}$) and sum this with the above equality to get
\begin{equation}\label{eq1}
a_3u_i+a_{-3}u_i=u_i-\tfrac{1}{2}({u_i}^{\ta}+{u_i}^{\ta\tau_0})-6(t_{3,i}+z_{3,i}).
\end{equation}
Similarly, multiply $u_i$ by $s_3+z_3$ using Lemmas~\ref{prodz}, \ref{transition}, and \ref{preliminary}, and we get
\begin{align*}
(s_{3}+z_3) u_i &= (s_{3}+z_3)\left(3c_i + 4s_i+4z_{\0,i}\right) \\
&= \tfrac{9}{8} c_{3,i} -\tfrac{3}{2}(t_{3,i}+z_{3,i}) \\
&= -\tfrac{3}{4} u_i + \tfrac{3}{8}(u_i^{\ta} + u_i^{\ta\tau_0})  -\tfrac{3}{2}(t_{3,i}+z_{3,i}) .
\end{align*}
We sum this last equation with $\frac{3}{4}$ of Equation~(\ref{eq1}) to obtain
\begin{equation}\label{eq2}
t_{3,i}+z_{3,i}=-\tfrac{1}{6}s_{\0,3}u_i-\tfrac{1}{8}(a_3u_i+a_{-3}u_i).
\end{equation}
Finally, the result for $u_i$ follows by substituting this expression for $t_{3,i}+z_{3,i}$ in the expression for $a_3u_i$ and rearranging.
The proof for $v_i$ is obtained analogously, by taking a suitable linear combination of the expressions for $a_3v_i$, $(a_3 v_i)^{\tau_0}$ and $s_{3} v_i$.
\end{proof}

\begin{corollary}\label{cor2}
${I_2}^{\Aut(\hatH)} \subseteq I$.
\end{corollary}
\begin{proof}
Since $\hatH_2 = \hatH_v$, the result follows from Lemma~\ref{uf}(2).
\end{proof}

We now consider the case where $\lambda=0$. If $\ch(\FF) \neq 5$, then ${I_0}^{\Aut(\hatH)} \subseteq I$ follows immediately from the Lemma \ref{uf}(1).  However if $\ch(\FF) =5$, then $\sfrac{5}{2} = 0$, so the $0$-eigenspace is $\hatH_u\oplus \hatH_z$ and has basis given by the $u_i$'s and the $z_i$'s.

\begin{lemma}\label{5/2}
If $\ch(\FF) =5$, then, for all $x \in \hatH_u\oplus \hatH_z$, 
\[
{x}^{\ta}= x+2a_{-3}x+s_{3}x+z_3x.
\]
\end{lemma}
\begin{proof}
If $x \in  \hatH_u$, the result follows immediately by Lemma~\ref{uf} (note that $\frac{3}{4} = 2$ in characteristic $5$).  Whereas for $\hatH_z$, we have $a_{-3}z_i=(a_0z_i)^{\tau_{-\sfrac{3}{2}}}=0$, so, by Lemma~\ref{ss} and Lemma~\ref{ztw}, $z_i+2a_{-3}z_i+s_{3}z_i+z_{3}z_i = {z_i}^{\ta}+0+0+0 = {z_i}^{\ta}$. 
\end{proof}

\begin{corollary}\label{cor0}
${I_0}^{\Aut(\hatH)} \subseteq I$.
\end{corollary}
\begin{proof}
If $\ch(\FF) \neq 5$, then $I_0=I\cap \hatH_u$ and the result follows by Lemma~\ref{uf}(1). If $\ch(\FF) = 5$, the result follows by Lemma~\ref{5/2}.
\end{proof}

It now remains to consider the case where $\lm = \frac{1}{2}$.  Here, the $\frac{1}{2}$-eigenspace $\hatH_w$ has a basis given by two different types of vectors, $w_i = a_{-i} - a_i$, for $i \in \N$, and $\tw_j = p_{\1,j} + p_{\2,j}=-p_{\0,j}$, where $j \in 3\N$. We first compute ${w_i}^{\flip}$ (here in fact it is more convenient to use $\flip$, rather than $\ta$).

\begin{lemma}\label{wf}
For every $i\in \N$ we have
\[
{w_i}^{\flip}= \tfrac{4}{3}a_0(a_1w_i) - \tfrac{4}{3}s_{1}w_i - \tfrac{4}{3}w_i -2(a_1w_i - \tfrac{1}{2}w_i) 
+ \tfrac{4}{3}(a_1w_i - \tfrac{1}{2}w_i)^{\tau_2-\tau_2\tau_1}.
\]
\end{lemma}
\begin{proof}
We begin by calculating
\begin{align*}
s_1 w_i &= s_1 (a_{-i} - a_i) \\
&= -\tfrac{3}{4}w_i + \tfrac{3}{8}( a_{-i-1} + a_{-i+1} - a_{i-1} - a_{i+1})
\end{align*}
Now, noting that $\flip$ acts on indices of the $a_j$'s by $j \mapsto 1-j$ and $\flip\tau_0$ acts by translation by $-1$, we have ${w_i}^{\flip} = a_{i+1} - a_{-i+1}$ and ${w_i}^{\flip\tau_0} = a_{-i-1} - a_{i-1}$.  So the above is $s_1 w_i = -\tfrac{3}{4}w_i + \tfrac{3}{8}({w_i}^{\flip\tau_0} - {w_i}^{\flip})$.

Also by calculation $a_1w_i=\tfrac{1}{2}w_i+s_{i+1}-s_{|i-1|}+z_{\1,i+1}-z_{\1,|i-1|}  = \tfrac{1}{2}w_i+s_{i+1}-s_{i-1}+z_{\1,i+1}-z_{\1,i-1}$, as $i \geq 1$.  Multiplying by $a_0$, we get
\begin{align*}
a_0(a_1w_i) &= \tfrac{1}{2}a_0w_i + a_0(s_{i+1}-s_{i-1}+z_{\1,i+1}-z_{\1,i-1})\\
&= \tfrac{1}{4}w_i + \tfrac{3}{8}( a_{-i-1} + a_{i+1} - a_{-i+1} - a_{i-1}) + \tfrac{3}{2}(s_{i+1}-s_{i-1}) \\
&\phantom{{}+{}} - z_{\0, i+1}+z_{\0, i-1}+\tfrac{3}{2}z_{\1,i+1}+z_{\2,i+1}-\tfrac{3}{2}z_{\1,i-1}-z_{\2,i-1} \\
&= \tfrac{1}{4}w_i +\tfrac{3}{8}({w_i}^{\flip\tau_0} + {w_i}^{\flip}) + \tfrac{3}{2}(s_{i+1}-s_{i-1})\\
&\phantom{{}+{}} +\tfrac{3}{2}(z_{\1,i+1}-z_{\1,i-1})- (z_{\0, i+1}-z_{\0, i-1}-z_{\2,i+1}+z_{\2,i-1})
\end{align*}
Note that $s_{i+1}-s_{i-1}+z_{\0,i+1}-z_{\0,i-1} = (s_{i+1}-s_{i-1}+z_{\1,i+1}-z_{\1,i-1})^{\tau_2} = (a_1w_i-\tfrac{1}{2}w_i)^{\tau_2}$ and similarly, $s_{i+1}-s_{i-1}+z_{\2,i+1}-z_{\2,i-1} = (a_1w_i-\tfrac{1}{2}w_i)^{\tau_2\tau_1}$, so that $z_{\0, i+1}-z_{\0, i-1}-z_{\2,i+1}+z_{\2,i-1}=(a_1w_i-\tfrac{1}{2}w_i)^{\tau_2}-(a_1w_i-\tfrac{1}{2}w_i)^{\tau_2\tau_1}$.  We can now combine these two expressions with those for $s_{1} w_i$ and $a_0(a_1w_i)$ to get the result.
\end{proof}

When $\ch(\FF)=5$, the formula in Lemma~\ref{wf} holds also for the $\tw_i$'s:

\begin{lemma}\label{wf2}
Suppose that $\ch(\FF)=5$. Then, for every $i\in \N$, we have
\[
{\tw_i}^{\flip}=\tfrac{4}{3}a_0(a_1\tw_i) - \tfrac{4}{3}s_{1}\tw_i - \tfrac{4}{3}\tw_i -2(a_1\tw_i - \tfrac{1}{2}\tw_i)
+ \tfrac{4}{3}(a_1\tw_i - \tfrac{1}{2}\tw_i)^{\tau_2-\tau_2\tau_1}.
\]
\end{lemma}
\begin{proof}
Since $\tw_i = -p_{\0,i}$, by H3, we have
\[
a_1\tw_i = -a_1p_{\0,i}=-\tfrac{3}{2} p_{\0,i}  +p_{\2,i}= -\tfrac{5}{2} p_{\0,i} -p_{\1,i}= -p_{\1,i}= - {\tw_i}^{\flip}
\]
and hence
\[
a_0(a_1\tw_i) = - a_0{\tw_i}^{\flip}=  - (a_1 {\tw_i})^{\flip}=  - (- {\tw_i}^{\flip})^{\flip} =\tw_i
\]
By H5, we get $s_{1}\tw_i = -s_1p_{\0,i} = -\tfrac{3}{4} p_{\0,i} + \tfrac{3}{8}(0)= \tfrac{3}{4} \tw_i$.  

Now observe that $\tau_2$ and $\flip$ have the same action on $S_i:= \la s_{i}, p_{\1,i}, p_{\2,i} \ra$.  Hence $\tau_2\tau_1$ and $\flip \tau_1  = \tau_0 \flip$ also have the same action on $S_i$. In particular, noting that $\tw_i^{\flip} = p_{\1,i}$ and so $\tw_i^{\tau_1} = -p_{\0,i}^{\tau_1} = p_{\2,i} = \tw_i - \tw_i^{\flip}$, we have
\begin{align*}
(a_1\tw_i - \tfrac{1}{2}\tw_i)^{\tau_2-\tau_2\tau_1} &= (-\tw_i^{\flip} - \tfrac{1}{2}\tw_i)^\flip - (-\tw_i^{\flip} - \tfrac{1}{2}\tw_i)^{\flip\tau_1} \\
&=-\tw_i - \tfrac{1}{2}\tw_i^{\flip} - (-\tw_i^{\tau_1} - \tfrac{1}{2}\tw_i^{\tau_0 \flip}) \\
&= -\tw_i - \tfrac{1}{2}\tw_i^{\flip} - (- \tw_i + \tw_i^{\flip} + \tfrac{1}{2}\tw_i^\flip) \\
&= -2\tw_i^\flip
\end{align*}

Thus, the right hand side of the required expression is
\begin{align*}
& \tfrac{4}{3}a_0(a_1 \tw_i) - \tfrac{4}{3}s_{1}\tw_i - \tfrac{4}{3}\tw_i -2(a_1\tw_i - \tfrac{1}{2}\tw_i) + \tfrac{4}{3}(a_1\tw_i - \tfrac{1}{2}\tw_i)^{\tau_2-\tau_2\tau_1} \\
&\qquad= \tfrac{4}{3}\tw_i  - \tw_i - \tfrac{4}{3}\tw_i -2(-\tw_i^{\flip} - \tfrac{1}{2}\tw_i) - \tfrac{8}{3}\tw_i^\flip\\
&\qquad= (2 - \tfrac{8}{3})\tw_i^\flip = \tw_i^\flip. \qedhere
\end{align*}
\end{proof}

\begin{corollary}\label{ch5w}
If $\ch(\FF)=5$, then ${I_{\frac{1}{2}}}^{\Aut{\hatH}} \subseteq I$.
\end{corollary}
\begin{proof}
The result follows from Lemma~\ref{wf} and Lemma~\ref{wf2}.
\end{proof}

When the characteristic is not $5$, we will show that we can in fact further decompose $I_\frac{1}{2}$ as $I_\frac{1}{2}=I_w \oplus I_{\tw}$, where $I_w := I \cap \la w_i : i \in \N \ra$ and $I_{\tw} := I \cap \la \tw_j : j \in 3\N \ra$. 
Recall that, for $k \in \Z$, $ \theta_k = (\tau_0 \flip)^k$  and $\theta_k$ maps $a_i$ to $a_{i+k}$ for all $i\in \Z$, fixes $s_j$ and maps $p_{\r,j}$ to $p_{\r+\overline k,j}$, for all $j\in \N$ and $\r\in \Z_3$.  In particular, $\theta_{2k} =  \theta_k  \theta_k\in \Miy(X)$.

\begin{lemma}\label{more formulas}
For every $i\in \N$, $j\in 3\N$, $k \in \N \setminus 3\N$ we have
\begin{enumerate}
\item $\tw_{j}+{\tw_j}^{\theta_2}+{\tw_j}^{\theta_4}=0$;
\item if $i\not \in 3\N$, $s_{i}\tw_j=\tfrac{3}{4}\tilde w_j$;
\item $s_{k} w_i = -\frac{3}{4} w_i + (w_i^{\theta_k} + w_i^{\theta_{-k}})$.
\end{enumerate}
\end{lemma}
\begin{proof}
This is immediate from the definitions.  
\end{proof}

\begin{lemma}\label{separate}
If $\ch(\FF)\neq 5$, then $I_\frac{1}{2} = I_w \oplus I_{\tw}$. 
\end{lemma}
\begin{proof}
Let $x \in I_\frac{1}{2}$ and write $x = w+ \tw$, where $w \in \la w_i : i \in \N \ra$ and $\tw \in \la \tw_j : j \in 3\N \ra$.  We must show that $w, \tw \in I$.  By Lemma~\ref{more formulas}, we have
\begin{align*}
x^{\theta_2} + x^{\theta_{-2}} + x^{\theta_4} + x^{\theta_{-4}} &= w^{\theta_2} + w^{\theta_{-2}} + w^{\theta_4} + w^{\theta_{-4}} \\
&\phantom{{}={}} + \tw^{\theta_2} + \tw^{\theta_{-2}} + \tw^{\theta_4} + \tw^{\theta_{-4}} \\
&= \tfrac{3}{2}w + (s_{2} + s_{4})w - 2 \tw \\
&= \tfrac{3}{2}w + (s_{2} + s_{4})(w + \tw)  -\tfrac{3}{2} \tw -2\tw \\
&= (s_{2} + s_{4})(w + \tw) + \tfrac{3}{2}w -\tfrac{7}{2} \tw
\end{align*}
Since $I$ is invariant under the Miyamoto group and $\theta_{2k} \in \Miy(X)$, for $k \in \Z$, we have $x^{\theta_2} + x^{\theta_{-2}} + x^{\theta_4} + x^{\theta_{-4}} - (s_{2} + s_{4})x = \frac{3}{2}w -\tfrac{7}{2} \tw \in I$.  Hence, $5 w = \frac{3}{2}w -\tfrac{7}{2} \tw + \tfrac{7}{2}(w+\tw) \in I$.  Therefore, since the characteristic is not $5$, $w$ and hence $\tw$ are both in $I$.
\end{proof}

\begin{corollary}\label{cor1/2}
If $\ch(\FF)\neq 5$, then $I_\frac{1}{2}^{\Aut{\hatH}} \subseteq I$.
\end{corollary}
\begin{proof}
By Lemma~\ref{separate}, we may decompose $I_\frac{1}{2}$ as $I_w \oplus I_{\tw}$.  By Lemmas~\ref{wf} and \ref{ztw}(2), $I_w^{\Aut{\hatH}}$ and $I_{\tw}^{\Aut{\hatH}}$ are both in $I$.
\end{proof}

\begin{proof}[Proof of Theorem $\ref{symmetry}$]
The result follows from  Corollaries~\ref{cor5/2}, \ref{cor2}, \ref{cor0}, \ref{ch5w}, and \ref{cor1/2}.
\end{proof}

%%%%%%%%%%%%%%%%%%%%%%%%%%%

\section{Ideals in $J$}\label{sec:infinite}

In this section, we characterise ideals contained in $J$.
Every element $x\in J$ can be written in a unique way as
\[
x = \sum_{\substack{j=1 \\ \r \in \{\1,\2\}}}^{3k} \bt_{\r,j} p_{\r,j}
\]
with $\bt_{\r,j} \in \FF$ and $\bt_{\r,3k}\neq 0$ for some $\r\in \{\1,\2\}$. We define the \emph{ $p$-level} of $x$ to be $l_p(x) := 3k$ and $\bt_{\1,3k} p_{\1,3k}+\bt_{\2,3k} p_{\2,3k}$ to be the \emph{tail} of $x$. Furthermore, for $x \in J$ of $p$-level $3k$
we define the \emph{$J$-degree} of $x$ as
\[
\deg_J(x) := 3k + \sum_{\substack{r \in \{1,2\}\\ \bt_{\r,3k}\neq 0}} \tfrac{r}{4}
\]
so $\deg_J(p_{\1,3k}) = 3k+\sfrac{1}{4}$, $\deg_J(p_{\2,3k}) = 3k+ \sfrac{1}{2}$, and $\deg_J(p_{\1,3k} + p_{\2,3k}) = 3k+\sfrac{3}{4}$.  In particular, $l_p(x) = \lfloor \deg_J(x) \rfloor$, for $x \in J$.

Note that the $J$-degree induces a total order on the set $\{p_{\r,j} : j\in 3\N, \r\in \{\1,\2\}\}$.

\begin{theorem}\label{Jideals}
There is a bijection between the set of tuples ${(\bt_3, \dots,\bt_{3k}) \in \FF^k}$, for $k\in \N$, up to scalar multiples, and the ideals $I \subseteq J$, given by
\[
 (\bt_3, \dots,\bt_{3k})\mapsto (x), \:\: \mbox{ where } x:=\sum_{j=1}^k \bt_{3j}p_{\1,3j} 
\]
and the inverse is given by taking the tuple of coefficients of an element of minimal $J$-degree.  In particular, all ideals $I \subseteq J$ are principal.
\end{theorem}

This theorem will follow from the next theorem.

\begin{theorem}\label{Jidealbasis}
Let $x := \sum_{j=1}^k \bt_{3j}p_{\1,3j}$.  Then $I = (x)$ has basis given by
\[
x, x^{\tau_0}, s_{ i} x, (s_{ i} x)^{\tau_0}
\]
for all $i \in 3\N$.
\end{theorem}

Note that, after scaling, $x$ has tail $p_{\1,3k}$, $x^{\tau_0}$ has tail $p_{\2,3k}$, $s_i x$ has tail $p_{\1,3k+i}$ and $(s_i x)^{\tau_0}$ has tail $p_{\2,3k+i}$, for $i \in 3\N$.  So we have an immediate corollary.

\begin{corollary}\label{Jcodimension}
Let $I$ be a non-zero ideal of $\hatH$ contained in $J$ and let $x$ be a non-zero element of minimal $J$-degree in $I$. Then $I$ has codimension $2(k-1)$ in $J$, where $l_p(x) = 3k$. 
\end{corollary}

\begin{proof}
By Theorem \ref{Jideals}, $I$ is generated by $x$, which we may assume has tail $p_{\1,3k}$ and now by Theorem \ref{Jidealbasis}, the image of  $\{ p_{\1,3h}, p_{\2,3h} : 1 \leq h \leq k-1\}$ is a basis for $J/I$.
\end{proof}

We now prove the above two theorems via a series of lemmas, beginning with Theorem \ref{Jidealbasis}.

It is clear from the definition of the $J$-degree that in every ideal $I \subseteq J$, there is a unique element $x$, up to scaling, of minimal $J$-degree.

\begin{lemma}\label{minJdegree}
Let $x$ be an element of minimal $J$-degree in $I \subseteq J$.  Then, $x = \sum_{j=3}^{3k} \bt_{j}p_{\1,j}$ for some $\bt_{j} \in \FF$.
\end{lemma}
\begin{proof}
We may write $x = \sum_{j=1,\, \r \in \{\1,\2\}}^{3k} \bt_{\r,j} p_{\r,j}$, for some $\bt_{\r,j}\in \FF$.  First, we claim that, for the tail, $\bt_{\2, 3k} = 0$.  For a contradiction suppose not.  If $\bt_{\1,3k} = 0$, then $x^{\tau_0}$ has tail $-\bt_{\2, 3k}p_{\1,3k}$ and so has lower $J$-degree than $x$, a contradiction.  So suppose that $\bt_{\1,3k}$ and $\bt_{\2,3k}$ are both non-zero.  Then  $x^{\tau_1}$ has tail $-\bt_{\1,3k}p_{\1,3k} + \bt_{\2,3k}(p_{\1,3k} + p_{\2,3k}) = (\bt_{\2,3k} - \bt_{\1,3k})p_{\1,3k} + \bt_{\2,3k}p_{\2, 3k}$ and so $\bt_{\2,3k}^{-1}(x+x^{\tau_1})$ has tail $p_{\1,3k}+2p_{\2,3k}$.  Hence $\bt_{\2,3k}^{-1}(x+x^{\tau_1})+2\bt_{\2,3k}^{-1}(x+x^{\tau_1})^{\ta}$ has tail $-3p_{\1,3k}$, a contradiction as above.  So $\bt_{\2, 3k} = 0$.

Now suppose there exists $j\in 3\N$ such that $\bt_{\2,j} \neq 0$.  Then, $x^{\tau_1}$ has tail $-\bt_{\1,k} p_{\1,k}$, and, similarly to above, its level $j$ part is $(\bt_{\2, j}-\bt_{\1,j})p_{\1,j} + \bt_{\2,j}p_{\2,j}$.  So $0\neq x+x^{\tau_1}\in I$ has $J$-degree strictly less than $x$, a contradiction.
\end{proof}

Before proving the theorem, we need the following lemma.

\begin{lemma}\label{ssp}
For all $i,j,k \in 3\N$, $\a, \b \in \Z_3$, we have
\begin{align*}
s_{k}(s_{j} p_{\a,i}) &= (s_{k}s_{j})p_{\a,i} \\ 
p_{\b, k}(s_{j} p_{\a,i}) &= (p_{\b,k}s_{j})p_{\a,i}
\end{align*}
\end{lemma}
\begin{proof}
We prove the second of these.  The first follows from an analogous, but easier argument.  Let $\c = -(\a+\b)$. By H5, we have 
\begin{align*}
\tfrac{8^2}{3}p_{\b, k}(s_{j} p_{\a,i}) &= 8p_{\b, k} \left[ 2(p_{\a,i} + p_{\a,j}) - (p_{\a,|i-j|} + p_{\a,i+j}) \right] \\
&= 2 \left[ 8p_{\b, k}p_{\a,i} + 2(z_{\c,k} + z_{\c,j}) - (z_{\c,|j-k|} + z_{\c,j+k}) \right] \\
&\phantom{{}={}} - \left[ 2(z_{\c,k} + z_{\c,|i-j|}) - ( z_{\c,||i-j| -k|} + z_{\c,|i-j|+k})  \right.\\
&\phantom{{}= -[{}} \left.  + 2(z_{\c,k} + z_{\c,i+j}) - (z_{\c,|i+j-k|} +z_{\c,i+j+k}) \right] \\
&= 2\left[ 8p_{\b, k}p_{\a,i}+ 2(z_{\c,j} + z_{\c,i}) - ( z_{\c,|i-j|} + z_{\c,i+j}) \right] \\
&\phantom{{}={}} - \left[ 2(z_{\c,|j-k|} + z_{\c,i}) + 2(z_{\c,j+k} + z_{\c,i}) \right. \\
&\phantom{{}= -[{}} \left. - (z_{\c,||i-j| -k|} + z_{\c,|i-j|+k} + z_{\c,|i+j-k|} + z_{\c,i+j+k}) \right]\\
&= 8\big(2(p_{\b,k} + p_{\b,j})\big)p_{\a,i} - \left[ 2(z_{\c,|j-k|} + z_{\c,i}) + 2(z_{\c,j+k} + z_{\c,i}) \right. \\
&\phantom{{}= -[{}} \left. - (z_{\c,||i-j| -k|} + z_{\c,|i-j|+k} + z_{\c,|i+j-k|} + z_{\c,i+j+k}) \right]
\end{align*}
If the sum in the square brackets is equal to $8(p_{\b,|j-k|} + p_{\b,j+k})p_{\a,i}$, then the above is equal to $8\left( 2(p_{\b,k} + p_{\b,j}) - p_{\b,|j-k|} + p_{\b,j+k}\right)p_{\a,i}$ and hence the result follows.  This is equivalent to the two sets $A:=\{ ||i-j| -k|, |i-j|+k ,|i+j-k|, i+j+k \}$ and $B:=\{ ||j-k| -i|, |j-k| + i, |j+k-i|, j+k+i \}$ being equal.  Since $s_{j}p_{\a,i}$ is symmetric in $i$ and $j$, without loss of generality we may assume that $i-j\geq 0$.  The result now follows after a case analysis on the parity of $j-k$.
\end{proof}

\begin{proof}[Proof of Theorem $\ref{Jidealbasis}$]
Let $B = \{ x, x^{\tau_0}, s_{3i}x, (s_{3i}x)^{\tau_0} : i \in \N \}$.  First, note that by H5, $s_{3i}x$ has $J$-degree $3(i+k) + \sfrac{1}{4}$ and so $(s_{3i}x)^{\tau_0}$ has $J$-degree $3(i+k) +  \sfrac{1}{2}$.  So it is clear that $B$ is a linearly independent set. Moreover, since by Theorem~\ref{symmetry} ideals of $\hatH$ are invariant under $\Aut(\hatH)$, it is clear that $B \subseteq (x)$.   

So to show that $B \subseteq (x)$, it suffices to show that $\la B \ra$ is closed under multiplication by $\hatH$ and so is an ideal.  Since $\la B \ra$ is clearly invariant under the action of $\Aut(\hatH)$, it is enough to show that $a_j x$, $s_{k} x$, $p_{\r,k}x$, $a_j(s_{3i} x)$, $s_{k}(s_{3i} x)$ and $p_{\r,k}(s_{3i} x)$ are in $\la B \ra$, for all $j \in \Z$, $i,k \in \N$, $\r \in \{ \1,\2\}$.
By H2, for $l \in 3\N$, we have
\[
a_jp_{\1,l}=\tfrac{3}{2} p_{\1,l}-p_{-\1-\jj,l}=\tfrac{3}{2} p_{\1,l}+p_{\1,l}^{\tau_j}
\]
and hence $a_j x = \tfrac{3}{2} x + {x}^{\tau_j} \in \la B \ra$.  Similarly, $a_j(s_{3i} x) = \tfrac{3}{2}(s_{3i} x) + (s_{3i} x)^{\tau_j} \in \la B \ra$.  For $s_k$, note that if $k \notin 3\N$, then by H5, $s_{ k}x = \tfrac{3}{4}x \in \la B \ra$.  Again similarly, $s_{ k}(s_{3i}x) = \tfrac{3}{4}(s_{3i}x) \in \la B \ra$.  If $k \in 3\N$, then by definition, $s_kx \in \la B \ra$.  By Lemma \ref{ssp}, $s_{3k}(s_{3i} x)=(s_{3k}s_{3i}) x \in \la B \ra$, for all $k \in \N$. Now we consider $p_{\r,k}$.  For $l \in 3\N$, note that $z_{-(\r+\1), l} = p_{-\r,l} - p_{1-\r,l} = -p_{\1,l}^{\tau_{(1-r)/2} - \tau_{(2-r)/2}}$.  Now, by H6, for $k,l \in 3\N$, we have
\begin{align*}
p_{\r,k} p_{\1,l} &= \tfrac{1}{4}(z_{-(\r+\1), k}+ z_{-(\r+\1),  l}) - \tfrac{1}{8}(z_{-(\r+\1), |k-l|}+z_{-(\r+\1), k+l}) \\
&= - \tfrac{1}{3}(s_{k}p_{\1,l})^{\tau_{(1-r)/2} - \tau_{(2-r)/2}}
\end{align*}
Hence, $p_{\r,k}x = - \tfrac{1}{3}(s_{k}x)^{\tau_{(1-r)/2} - \tau_{(2-r)/2}} \in \la B \ra$.  Finally, by Lemma \ref{ssp}, $p_{\r,k}(s_{3i} x)=(p_{\r,k}s_{3i}) x$, for all $k \in 3\N$, which is in $\la B \ra$ by H5 and the above results. Therefore, $\la B \ra$ is closed under multiplication by $\hatH$ and hence $B$ is a basis for the ideal $(x)$.
\end{proof}

We can now complete the proof of the remaining theorem.

\begin{proof}[Proof of Theorem $\ref{Jideals}$]
Let $I \subseteq J$.  Then $I$ contains an element $x$ of minimal $J$-degree which is unique up to scaling.  It is clear that $(x) \subseteq I$, so we must show that $I = (x)$.  Suppose for a contradiction $0 \neq y \in I \setminus (x)$.  By Lemma~\ref{minJdegree}, $x = \sum_{j=1}^k \bt_{3j}p_{\1,3j}$, for some $\bt_{3j} \in \FF$, and we may scale so that $\bt_{3k} = 1$.  Since $x$ has minimal $J$-degree in $I$, $y$ has $J$-degree strictly greater than $x$.  Now, by Theorem \ref{Jidealbasis}, $(x)$ has a basis $B := \{ x, x^{\tau_0}, s_{3i}x, (s_{3i} x)^{\tau_0} : i \in \N \}$.  Note that the tails of the elements in $B$ are $p_{\1,3k}, p_{\2,3k}, p_{\1,3(k+i)}, p_{\2,3(k+i)}$, respectively.  So by taking a suitable linear combination $b$ of elements of $B$, we obtain an element $z := y - b \in I$ with $J$-degree strictly less than that of $x$.  Since $y \notin (x)$, $z \neq 0$, which is a contradiction.  Hence $I = (x)$ as claimed. 
\end{proof}

We close this section with an observation which will be used in Section~\ref{sec:principal}.

\begin{lemma}\label{Jminimal}
Let $I$ be a non-zero ideal of $\hatH$ contained in $J$ and let $x$ be a non-zero element of minimal $J$-degree in $I$. Then $I=(x^\prime)$ for every element $x^\prime$ of $I$ with the same $p$-level as $x$.
\end{lemma}
\begin{proof}
Suppose that $x$ has tail $\bt_{\1,3k}p_{\1,3k}$ and let $x^\prime$ be an element of $I$ of $p$-level $3k$. By arguing as in the proof of Lemma~\ref{minJdegree}, we see that $(x^\prime)$ contains an element $x^{\prime \prime}$ with tail $\bt_{\1,3k}p_{\1,3k}$. Then $x-x^{\prime \prime}$ has $p$-level at most $3(k-1)$ and the minimality of $x$ implies $x=x^{\prime \prime}$, whence $I=(x)=(x^{\prime\prime})\subseteq (x^\prime)\subseteq I$. 
\end{proof}

%%%%%%%%%%%%%%%%%%%%%%%%%%%%%%%%%

\section{Ideals are principal} \label{sec:principal}

Our goal for this section is to prove the following.

\begin{theorem}\label{principal}
Every ideal in $\hatH$ is principal.
\end{theorem}

We already showed in the previous section that ideals that are contained in $J$ are principal.  So for the remainder of this section, let $I$ be an ideal of $\hatH$ which is not contained in $J$.

We will choose a nice candidate $y$ for a generator of the ideal $I$ and then use a sort of Euclidean division algorithm to show that every other element $x \in I$ is in fact in $(y)$.

We begin by defining a partial order on $\hatH$, which we will use to define our candidate $y$.  
Every element $x\in \hatH$ can be written in a unique way as
\[
x=x_a+x_s+x_p
\]
where $x_a\in \langle a_i\mid i\in \Z\rangle$, $x_s\in \langle s_{i}\mid i\in \N\rangle$, and $x_p\in \langle p_{\r,i}\mid i\in 3\N, \r\in \{1,2\}\rangle$. We call $x_a$ the \emph{$a$-part} of $x$,  $x_s$ the \emph{$s$-part} of $x$, and $x_p$ the \emph{$p$-part} of $x$. Finally, we call $x_s+x_p$ the \emph{$a^\prime$-part} of $x$.
We define the \emph{$a$-length}, or just \emph{length}, of $x$ to be $l_a(x) = m-l+1$, where $x_a = \sum_{i=l}^m \al_i a_i$ and $\al_l \neq 0 \neq \al_m$.  Similarly, if 
\begin{equation}
\label{dec}
x_s:=\sum_{j=1}^k \beta_{j} s_{j}, \:\:\:x_p:= \sum_{j=1, \r=\1, \2}^l \beta_{\r,3j} p_{\r,3j}
\end{equation}
we define the $s$-level of $x$ to be $l_s(x) := \max\{ j \in \N : \bt_j \neq 0 \}$ and we have already defined the $p$-level of $x$ to be $l_p(x) := \max\{ j \in \N : \bt_{\r,j} \neq 0 , \mbox{ for some } \r = \1,\2 \}$.
If $n = \max\{l_s(x),l_p(x)\}$, then we call 
\[
\beta_{n}s_{n}+\beta_{\1,n}p_{\1,n}+\beta_{\2,n}p_{\2,n}
\]
the \emph{tail of $x$}.

We can now define a partial order on $\hatH$ by setting 
\[
x \leq y \quad \Longleftrightarrow \quad (l_a(x), l_s(x), l_p(x))\leq (l_a(y), l_s(y), l_p(y))
\]
with respect to the lexicographic order on $\Z\times \Z\times \Z$.  The following lemma is immediate.

\begin{lemma}\label{deginvariant}
$l_a$, $l_s$, $l_p$ and so $\leq$ are invariant under the action of $\Aut(\hatH)$.
\end{lemma}

A \emph{minimal element} of $I$ is a non-zero element of $I$ minimal with respect to the order $\leq$. An element is called \emph{$a$-minimal} if its $a$-part is non-trivial and it is minimal (with respect to $\leq$) with this property.  An element is called \emph{$as$-minimal} if its $a$-part and $s$-part are both non-trivial and it is minimal with this property.  An element is called \emph{pure $a$-minimal} if it has non trivial $a$-part, trivial $a^\prime$-part and it is minimal with this property.
Note that, by the above Lemma \ref{deginvariant}, being minimal, or (pure) $a$-minimal, or $as$-minimal is $\Aut(\hatH)$-invariant.

We will now see that $I$ contains elements with non-trivial $a$-part and so, in particular, $a$-minimal elements of $I$ exist.

\begin{lemma}\label{notcontained}
Every ideal $I$ of $\hatH$ not contained in $J$ contains an element with non-zero $a$-part and trivial $a^\prime$-part.
\end{lemma}
\begin{proof}
Let us show first that $I$ contains an element with non-trivial $a$-part. Let $x\in I\setminus J$. If $x$ has non-trivial $a$-part, we are done. Otherwise $x = x_s + x_p$, where $x_s \neq 0$ as $x \notin J$.  So, $x+x^{\theta_1} + x^{\theta_2} = 3x_s \in I$.  Now, by H2, we see that $y := a_0 x_s \in I$ has non-trivial $a$-part.  Since $\theta_3$ fixes the $s$-part and the $p$-part of $y$ and maps $a_i$ to $a_{i+3}$, for all $i\in \Z$, $y-y^{\theta_3} \in I$ has non-trivial $a$-part and trivial $s$-part and $p$-part.
\end{proof}
 
We now want to see that $as$-minimal elements exist.  To do this we prove the Folding Lemma which will also be useful in later sections.  Here and from now on we adopt the following useful notation.

{\bf Notation:} Where we have a sum of elements $x_a = \sum_{i=l}^m \al_i a_i$, for example, we may ease notation} and write $x_a = \sum_{i \in \Z} \al_i a_i$ instead by adopting the convention that $\al_i := 0$ for $i < l$ and $i > m$.  We also do this for sums of $s_j$, or $p_{\r,j}$.  Note however that any sum is still always finite.

\begin{lemma}[Folding Lemma]\label{folding}
Let $x = \sum_{i \in \Z} \al_i a_i$.  For $k \in \Z$, we have
\[
a_kx - \tfrac{1}{2} x = \sum_{i \in \Z} \al_i (s_{|i-k|}+z_{\kk,|i-k|} ) = \sum_{i \in \N} (\al_{k-i} + \al_{k+i}) (s_{i}+z_{\kk,i} )
\]
and so
\begin{align*}
(a_kx - \tfrac{1}{2} x)^{1+\theta_1+\theta_2} &= 3 \sum_{i \in \N} (\al_{k-i} + \al_{k+i}) s_{i} \\
(a_kx - \tfrac{1}{2} x)^{\theta_{-1} - \theta_1} &= 3 \sum_{i \in \N} (\al_{k-i} + \al_{k+i}) p_{\kk,i}
\end{align*}
\end{lemma}
\begin{proof}
By Lemma \ref{coeffsum}, $\sum_{i \in \Z} \al_i = 0$ and so, by H1, we get
\begin{align*}
a_k x - \tfrac{1}{2} x &= \sum_{i \in \Z} \al_i (s_{|i-k|}+z_{\kk,|i-k|} ) \\
&= \sum_{i < k} \al_i (s_{ k-i}+z_{\kk,{ k-i}} ) + \sum_{i > k} \al_i (s_{ i-k}+z_{\kk,{ i-k}} ) \\
&= \sum_{j \in \N} \al_{k-j} (s_{j}+z_{\kk,j} ) + \sum_{j \in \N} \al_{k+j} (s_{j}+z_{\kk,j} )\\
&= \sum_{j \in \N} (\al_{k-j} + \al_{k+j}) (s_{j}+z_{\kk,j} ).
\end{align*}
 Recall that $s_j$ is fixed by the action of $\Aut(\hatH)$.  Now, since ${z_{\r,j}}^{1+\theta_1+\theta_2} = z_{\r,j} + z_{\r+\1,j} + z_{\r+\2,j} = 0$ and ${z_{\r,j}}^{\theta_{-1}-\theta_1} = z_{\r-\1,j} - z_{\r+\1,j} = 3p_{\r,j}$, the results follow.
\end{proof}

So by Lemmas~\ref{notcontained} and \ref{folding}, there exists elements $x \in I$ with non-trivial $a$- and $s$-parts and hence $I$ contains $as$-minimal elements.  Also by the above two lemmas, note that $I \cap J \neq 0$.

\begin{lemma}\label{minimal a-part}
Let $y$ be an $as$-minimal element of $I$.  If $x \in I$ has non-zero $a$-part, then $l_a(x) \geq l_a(y)$. 
\end{lemma}
\begin{proof} 
Suppose that $y$ is an $as$-minimal element and $l_a(x) < l_a(y)$.  By the $as$-minimality of $y$, $x_s= 0$. Let $y=\sum_{i=l}^m \al_i a_i+y_s + y_p$. Since $I$ is $\Aut(\hatH)$-invariant, we may assume that $x = \sum_{i=0}^n \bt_i a_i+x_p$.  Then $y - \frac{\al_m}{\bt_n}x^{\theta_{m-n}}$ has non-trivial $a$-part with length strictly less that $l_a(y)$ and non-trivial $s$-part (equal to $y_s$), a contradiction.
\end{proof}

So every $as$-minimal element is $a$-minimal.  In fact, the coefficients of the $a$-part of an $a$-minimal element satisfy precise conditions. The following is an adaptation of \cite[Lemma 2.2]{yabe}.

\begin{lemma}\label{minimal1}
Let $y \in I$ be $a$-minimal \textup{(}pure $a$-minimal\textup{)}, where $l_a(y)=D+1$.
\begin{enumerate}
\item If $x \in I$ is another $a$-minimal \textup{(}pure $a$-minimal\textup{)} element, then up to scaling and the action of $\Aut(\hatH)$, $x_a$ and $y_a$ are equal.
\item Suppose $y_a:=\sum_{i=0}^D\alpha_ia_i$.  Then there exists $\epsilon = \pm 1$ such that, for all $i\in \{0, \ldots , D\}$, $\alpha_i=\epsilon \alpha_{D-i}$.
\end{enumerate} 
\end{lemma}
\begin{proof}
 We prove the case where $y$ is an $a$-minimal element; the pure $a$-minimal case follows similarly.  To prove the first claim,  by scaling and using the action of $\Aut(\hatH)$, we may assume that $y_a:=\sum_{i=0}^D\alpha_ia_i$ and $x_a = \sum_{i=0}^D \bt_i a_i$, where $\alpha_D = \beta_D$.  Now $x-y$ has length strictly less than $D+1$.  So by minimality, $x_a-y_a = 0$ and the result follows.

Let $k:=\frac{D+1}{2}$; then $\tau_k$ is the reflection in $\Aut(\hatH)$ that maps $a_0$ to $a_D$.  So $y^{\tau_k}= \sum_{i=0}^D \al_{D-i} a_i+y_{a^\prime}^{\tau_k}$ is also an element of $I$ with length $D+1$ and thus, by the first part of the lemma, its $a$-part is a multiple of $x_a$.  So there exists $\epsilon \in \FF$ such that $\al_i = \epsilon \al_{D-i}$ for all $i =0, \dots, D$.  Hence $\al_0 = \epsilon \al_D = \epsilon^2 \al_0$ and $\epsilon = \pm 1$ as required.
\end{proof}

For an $a$-minimal (resp. pure $a$-minimal) $y$, with $y_a=\sum_{i=l}^m\al_ia_i$, define $\sg = \sg(y) := \tau_{\sfrac{(m-l)}{2}}$.  We can reword Lemma~\ref{minimal1} as saying that there exists $\epsilon = \pm 1$ such that $y_a^\sg = \epsilon y_a$.  In this case, we say that $y$ is of {$\epsilon$-type}.  Since every $a$-minimal $y$ is of $\epsilon$-type for the same value of $\epsilon$, we make the following definition.

\begin{definition}
We say $I$ is of $\epsilon$-type if $y$ is of $\epsilon$-type, for any $a$-minimal element $y \in I$.
\end{definition}

Let $y$ be $as$-minimal in $I$; we now consider the $p$-part of $y$.  By Lemma~\ref{folding}, $I\cap J\neq 0$ and by Theorem~\ref{Jideals}, $I\cap J$ is principal. Hence $I\cap J=(e)$ where $e=e_p$ has $p$-level $3h$, for some $h \in \N$.  Note that, by Theorem~\ref{Jidealbasis}, for all $z \in I \cap J$, $l_p(z) \geq 3h$.

\begin{lemma}\label{y_p min}
Let $y \in I$ be $as$-minimal.  Then $l_p(y) < 3h$.
\end{lemma}
\begin{proof}
By Theorem \ref{Jidealbasis}, for every $j \in 3\N$ such that $j \geq 3h$, there exist elements in $(e) = I \cap J$ with tail $p_{\r,j}$, for all $r \in \{1,2\}$.  It follows that $l_p(y) < 3h$ by $as$-minimality.
\end{proof}

\begin{corollary}\label{y_p}
Let $y \in I$ be an $a$-minimal $\epsilon$-type element.  Then $y_p^\sigma = \epsilon y_p$.  Furthermore, if $\epsilon = -1$, then $y_s \in I$.
\end{corollary}
\begin{proof}
By Lemma \ref{minimal1}, $y_a^\sg = \epsilon y_a$.  Define $z := y-\epsilon y^\sg$. Since $y_s^\sg = y_s$, $z= y_s - \epsilon y_s + y_p - \epsilon y_p^\sg \in I$.  If $\epsilon = 1$, then $z = y_p - y_p^\sg \in I \cap J$.  By Lemma~\ref{y_p min}, $l_p(z) \leq l_p(y_p) < 3h$.  So $z = 0$ and hence $y_p^\sg = y_p$.  If $\epsilon = -1$, then $z = 2y_s + y_p + y_p^\sg$.  Since $p_{\r,3j}+p_{\r,3j}^{\theta_1}+p_{\r,3j}^{\theta_2}=p_{\0,3j}+p_{\1,3j}+p_{\2,3j}=0$, for every $j\in \N$ and $r\in \{1,2\}$, we have $6y_s = z + z^{\theta_1} + z^{\theta_2} \in I$.  Hence $y_s \in I$ and $y_p + y_p^\sg \in I \cap J$.  A similar argument as above on the $p$-level gives $y_p^\sg = - y_p$ as required.
\end{proof}

By Lemmas~\ref{Jminimal} and \ref{deginvariant}, $I \cap J$ is generated by $e^g$ for any $g \in \Aut(\hatH)$.  In particular, given an $as$-minimal element $y$ and setting $\sg = \sg(y)$, we can always choose $e \in I \cap J$ so that $e^\sg \neq \epsilon e$.

\begin{definition}
An element $\bar{y} \in I$ is \emph{good}, if $\bar{y} = y+e$, where $y \in I$ is $as$-minimal and $(e) = I \cap J$ such that $e^{\sg(y)} \neq \epsilon e$.
\end{definition}

Note that $\bar{y}_a = y_a$, $\bar{y}_s = y_s$ and, as $\bar{y}$ is still $a$-minimal, $\sigma(\bar{y}) = \sigma(y)$. Moreover, every ideal $I$ not contained in $J$ contains a good element since it contains an $as$-minimal element. We will show that $\bar{y}$ generates $I$.  We begin with the following.

\begin{lemma}\label{e in I}
Let $\bar{y}=y+e \in I$ be good, with $y$ $as$-minimal and $(e) = I \cap J$ such that $e^{\sg(y)} \neq \epsilon e$.  Then $e, y \in (\bar{y})$.
\end{lemma}
\begin{proof}
Let $I$ be of $\epsilon$-type. Define $z := \bar{y} - \epsilon \bar{y}^\sg \in (\bar{y})$.  By Corollary~\ref{y_p}, $y_p = \epsilon y_p$, so $z = y_s - \epsilon y_s + e - \epsilon e^\sg$.  If $\epsilon = 1$, then $z = e - e^\sg \in (\bar{y}) \cap J$, which is non-zero by choice, and so $z$ generates $I \cap J$ by Lemma~\ref{Jminimal}.  Hence $e \in (\bar{y})$ and so $y = \bar{y}-e \in (\bar{y})$ also.  If $\epsilon = -1$, then $z = 2y_s + e + e^\sg \in (\bar{y})$. By Corollary~\ref{y_p}, $y_s=\bar y_s\in (\bar y)$ and hence by a similar argument to before we get $e,y \in (\bar{y})$.
\end{proof}

We now explore those elements of ideals which have non-trivial $s$-part.

\begin{lemma}\label{x_s elements}
Let $x \in \hatH$ with $x_s \neq 0$ and $l_s(x) = k$. Then, for each $j\in \N$ such that $j\geq k$, $(x)$ contains an element $x'$ such that $l_s(x') = j$.  Moreover, if $x$ has trivial $a$-part, then $x'$ does too.
\end{lemma}
\begin{proof}
Decompose $x = x_a + x_s + x_p$ and define $x' = s_{j-k}x$.  If $(x) = \hatH$, then the claim follows immediately.  So assume that $(x) \neq \hatH$; now by Corollary~\ref{coeffsum}, the sum of the coefficients of $x_a$ is zero.  Hence by H2, $s_{j-k}x_a$ has no $s$-part and as $J$ is an ideal $s_{j-k}x_p$ has no $s$-part. Then it is clear by H3 that $s_{j-k} x_s$ has $s$-level $j$.  Note that, as $\langle s_i, p_{\r,j} : i\in \N, j\in 3\N, r\in \{1, 2\}\rangle$ is a subalgebra of $\hatH$, if $x$ has trivial $a$-part, then $x'$ does too.
\end{proof}

We can now prove a first version of our Euclidean algorithm with respect to the $s$-level.

\begin{proposition}\label{Euclidean1}
Let $y \in \hatH$ such that $y_s \neq 0$. Then, for every $x\in \hatH$, there exist $q\in (y)$ and $r\in \hatH$ such that $x=q+r$ and $l_s(r)<l_s(y)$.
\end{proposition}
\begin{proof}
We proceed by induction on $l_s(x)$.  If $l_s(x)< l_s(y) \neq 0$, the claim is true with $q=0$ and $r=x$.  So suppose $l_s(x)\geq l_s(y)$.  By Lemma~\ref{x_s elements}, there exists $y' \in (y)$ such that $l(y'_s) = l_s(x)$.  So there exists $\lm \in \FF$ such that $l_s(x- \lm y') < l_s(x)$.  Hence, by the inductive hypothesis, there exist $q \in (y)$ and $r \in \hatH$ with $l_s(r) < l_s(y)$ such that 
\[
x - \lm y' = q + r.
\]
Now we see that $x = (q + \lm y') + r$ and $q + \lm y' \in (y)$ as required.
\end{proof}

With the above, we can now show another version of a Euclidean algorithm with respect to the $a$-length and the $s$-level.

\begin{proposition}\label{Euclidean2}
Let $y$ be an element with non-trivial $a$-part and non-trivial $s$-part. Then, for every $x \in \hatH$, there exist $q \in (y)$ and $r \in \hatH$ such that $x  =q+r$, $l_a(r)< l_a(y)$, and $l_s(r) < l_s(y)$.
\end{proposition}
\begin{proof}
Suppose first that $l_s(x) < l_s(y)$. We proceed by induction on $l_a(x)$.  If $l_a(x)<l_a(y)$, then the result is trivially true with $q=0$ and $r=x$.  So for the inductive step, assume that $l_a(x) \geq l_a(y)$ and that the result is true for $a$-length strictly less that $l_a(x)$. Suppose that $y_a = \sum_{i=l}^m \al_i a_i$ and $x_a = \sum_{i=k}^{n} \bt_i a_i$, where $\al_l, \al_m, \bt_k, \bt_n \neq 0$.  Since $a_m^{\theta_{m-n}} = a_n$, there exists $\lm \in \FF$, so that   
 $l_a(x -  \lm y) < l_a(x)$.  Hence, by the inductive hypothesis, there exist $q' \in (y)$ and $ r \in \hatH$ such that 
\[
x - \lm y = q' + r
\]
and $l_a(r) < l_a(y)$, $l_s(r) < l_s(y)$.  Therefore the claim holds with $q = q' + \lm y$ and $r = r$.

Finally, suppose that $l_s(x) \geq l_s(y)$. By Proposition~\ref{Euclidean1}, there exist $q' \in (y)$ and $r' \in \hatH$ such that $x = q' + r'$ and $l_s(r') < l_s(y)$.  Now, by the first part of the proof, there exist $q'' \in (y)$ and $r \in \hatH$ such that $r'=q''+r$ and the result holds with $q= q'+q''$ and $r = r$.
\end{proof}

The following proposition now completes the proof of Theorem \ref{principal}.

\begin{proposition}\label{good gen}
Any good element of a non-trivial proper ideal $I$ not contained in $J$ generates $I$.
\end{proposition}
\begin{proof}
Suppose $\bar{y}$ is a good element in $I$, where $y$ is an $as$-minimal element and $(e)=I\cap J$, and let $x \in I$.  By Lemma~\ref{e in I}, $e, y \in (\bar y)$.

By Proposition \ref{Euclidean2},  there exists $q \in (y)$ and $r \in \hatH$ such that $x = q+r$, where $l_a(r)<l_a(y)$ and $l_s(r) < l_s(y)$.  Then $r = x-q \in I$.  As $y$ is $as$-minimal, by Lemma~\ref{minimal a-part}, $l_a(r) = 0$ and so $r$ has trivial $a$-part.  If $l_s(r) \neq 0$, then by Lemma~\ref{x_s elements}, there exists $r' \in (r) \subseteq I$ such that $l_s(x') = l_s(y)$.  Moreover, as $r$ has trivial $a$-part, so does $r'$.  Then some linear combination of $y$ and $r'$ has minimal $a$-length, but $s$-level strictly less than that of $y$, contradicting the $as$-minimality of $y$.  So $l_s(r) = 0$ and hence $r \in I \cap J = (e) \subset (y)$.  Since $r \in (y)$, we have $x = q+r \in (y) \subseteq (\bar{y})$ as required.
\end{proof}

\begin{corollary}\label{finitecodimthm}
Let $I$ be an ideal of $\hatH$.  Then $I$ has finite codimension if and only if it is not contained in $J$.
\end{corollary}
\begin{proof}
Since $J$ has infinite codimension, it is clear that every ideal contained in it also has infinite codimension.  For the converse, let $I$ be a proper ideal of $\hatH$ not contained in $J$. By Proposition \ref{good gen}, $I$ is generated by a good element $\bar y=y+e$, where $y$ is an $as$-minimal element and $(e)=I\cap J$. Let $x\in \hatH$. By Lemma~\ref{Euclidean2}, $x=q+r$, where $q\in I$, $l_a(r)<l_a(y)$, and $l_s(r)<l_s(y)$. By Theorem \ref{Jidealbasis}, there exists $r' \in I \cap J$ such that $l_p(r-r')<l_p(e)$.  Setting $x' = r-r'$, we see that $x+I = x'+I$ and the result follows.
\end{proof}

%%%%%%%%%%%%%%%%%%%%%%%%%%%%%%%%%%%%%%%%%%%%%%%%%%%%%%%%%%%%%%%%%%%%%%%%%%

\section{Ideals not contained in $J$}\label{sec:ideals not in J}

In Section~\ref{sec:infinite}, we got a complete characterization of ideals contained in $J$. To get a similar characterization for ideals not contained in $J$ is much more difficult, since the picture is more complicated. Hence, in this section we classify ideals not contained in $J$ satisfying a certain minimality condition and also give an explicit basis for such ideals.

Let $I$ be an ideal of $\hatH$.  We define the \emph{axial codimension} of $I$ as the (possibly infinite) dimension of the subspace of $\hatH/I$ generated by the images of the $a_i$'s (note that this is precisely the \emph{axial dimension} of $\hatH/I$ defined in~\cite[Section 2.2]{yabe}).

If $I$ is not contained in $J$, then, by Theorem~\ref{finitecodimthm}, $I$ has finite codimension and so it has finite axial codimension.  Conversely, since $J$ has infinite axial codimension, if $I$ is contained in $J$, then $I$ has also infinite axial codimension.  Hence an ideal has finite axial codimension if and only if it is not contained in $J$.

\begin{lemma}\label{yodel}
Let $I$ be an ideal of $\hatH$ and assume $I$ contains an element $x:=\sum_{i=0}^D\al_ia_i$. Then $I$ has axial codimension at most $D$.
\end{lemma}

\begin{proof}
Since $I$ is $\Aut(\hatH)$-invariant, $\sum_{i=0}^D\al_ia_{i+j} \in I$ for all $j \in \Z$.  So, for all $k \in \Z$, there exists an element $a_k - \sum_{i=1}^D \bt_i a_i \in I$ for some $\bt_i \in \FF$.  Hence, the images of the axes in $\hatH/I$ span a subspace of dimension at most $D$.
\end{proof}

\begin{corollary}\label{yodel2}
Let $I$ be an ideal of axial codimension $D$, then $I$ contains a pure $a$-minimal element $x=\sum_{i=0}^D \al_i a_i$ with $\al_0 \neq 0 \neq \al_D$.
\end{corollary} 

\begin{proof}
By assumption, the images of $a_0,\ldots,a_D$ in $\hatH/I$ are linearly dependent, i.e.\ $I$ contains a non zero element $x=\sum_{i=0}^D \al_i a_i$. If either $\al_0$, or $\al_D$ were zero, then $x$ would have $a$-length strictly less than $D$ and so, by Lemma~\ref{yodel}, $I$ would have axial codimension strictly less than $D$, a contradiction.
\end{proof}

\begin{definition}
\label{axcod}
Let $I$ be an ideal of finite axial codimension $D$ in $\hatH$ and let $x=\sum_{i=0}^D \al_i a_i \in I$, where $\al_0 \neq 0 \neq \al_D$ be pure $a$-minimal.  
Then we say $I$ has \emph{\pattern} $(\al_0, \dots, \al_D)$.
\end{definition}

Such a pure $a$-minimal element $x$ in $I$ is of $\epsilon$-type, for some $\epsilon=\pm 1$ (cf. Lemma \ref{minimal1}).  It also satisfies $\sum_{i=0}^D \al_i = 0$.

\begin{definition}
A tuple $(\alpha_0, \ldots , \alpha_D)\in \FF^{D+1}$ is said to be of \emph{ideal type} if $\alpha_0 \neq 0 \neq \alpha_D$, $\sum_{i=0}^D \al_i = 0$ and $(\alpha_0, \ldots , \alpha_D)$ is of $\epsilon$-type, for $\epsilon = \pm 1$.
\end{definition}

Since any ideal $I$ with {\pattern } $(\al_0, \dots, \al_D)$ contains $x:=\sum_{i=0}^D \al_i a_i$, it must contain the ideal $(x)$ generated by $x$.  In other words, $(x)$ is the unique minimal ideal with {\pattern } $(\al_0, \dots, \al_D)$.  Hence we have the following theorem.

\begin{theorem}\label{classification I not in J}
For every $D\in \N$, there is a bijection between the set of ideal-type $(D+1)$-tuples $(\alpha_0, \ldots , \alpha_D)\in \FF^{D+1}$, up to scalars, and the set of minimal ideals of axial codimension $D$ of $\hatH$ given by
\[
(\alpha_0, \ldots , \alpha_D) \mapsto \left(\sum_{i=0}^D \al_i a_i\right).
\]
\end{theorem}

If $I'$ is a (non-minimal) ideal of {\pattern } $(\al_0, \dots, \al_D)$, it contains some minimal ideal $I$ with the same \pattern.  In particular, $I'$ corresponds to an ideal of the finite-dimensional algebra $\hatH/I$.

We now give an explicit basis for a minimal ideal with {\pattern } $(\al_0, \dots, \al_D)$, but first we introduce some notation.  If $\al:=(\al_0,\ldots,\al_D)$ is an ideal-type tuple, for $\r \in \Z_3$, we define
\[
\al_{\r} := \sum_{i\in \Z, \ \ii = \r} \al_i = \sum_{\ii = \r} \al_i
\]
Since $\sum_{i \in \Z} \al_i = 0$, we have $\al_{\0} + \al_{\1} + \al_{\2} = 0$.

\begin{theorem}\label{idealspan}
Let $I$ be a minimal ideal of $\hatH$ with {\pattern } $(\al_0, \dots, \al_D)$ and $x:=\sum_{i=0}^D \al_i a_i$.
\begin{enumerate}
\item If $\al_{\1} = \al_{\2} = 0$, then $I$ is spanned by 
\begin{align*}
x_k &:= x^{\theta_k} = \sum_{i \in \Z} \al_i a_{i+k} & \mbox{for } k \in \Z \\
y_k &:= \sum_{i \in \N} (\al_{k-i} + \al_{k+i})s_{ i}  & \mbox{for } k \leq \left\lfloor \tfrac{D}{2} \right\rfloor\\
p_k(\r) &:= \sum_{i \in \N} (\al_{k-i} + \al_{k+i})p_{\r, i}  & \mbox{for } k \leq \left\lfloor \tfrac{D}{2} \right\rfloor, r  \in \{1,2\}
\end{align*}
\item Otherwise, $J \subset I$ and so $I$ is spanned by the above $x_k$ and $y_k$ and all $p_{\1,j}, p_{\2,j}$, for $j \in 3\N$.
\end{enumerate}
The above elements are a basis unless $D$ is even and $\epsilon = -1$.  In that case, by restricting $k < \tfrac{D}{2}$ for $y_k$ and $p_k(\r)$ we get a basis.
\end{theorem}

Note that if $I$ is an ideal with {\pattern } $(\al_0, \dots, \al_D)$ where $\al_{\1} = \al_{\2} = 0$, then we could still have that $J \subset I$. We have the following immediate corollary.

\begin{corollary}\label{finitecodimcor}
Let $I$ be a minimal ideal of axial codimension $D$. Then $\hatH/I$ has dimension at most $D + \left\lfloor \frac{D}{2} \right\rfloor + 2\left\lfloor \frac{D}{6} \right\rfloor$.
\end{corollary}
\begin{proof} 
Let $k = \left\lfloor \frac{D}{2} \right\rfloor$.  By Theorem~\ref{idealspan}, $\hatH/I$ is spanned by the images of $a_1, \dots, a_D, \allowbreak s_1, \dots, s_k, p_{\r, 3}, \dots, p_{\r, k}$, for $\r = \1,\2$.
\end{proof}

We prove Theorem~\ref{idealspan} via a series of lemmas.  We will first show that all the above elements are indeed contained in the ideal generated by $x$, then we will show that the subspace spanned by them is an ideal.

Firstly, since $I$ is $\Aut(\hatH)$-invariant, it is immediate that $x_k =x^{\theta_k}$ is in $I$, for $k \in \Z$.  Secondly, by Lemma~\ref{folding}, $I$ contains $y_k$, $p_k(\1)$ and $p_k(\2)$ for all $k \in \Z$.  It remains to see when $J \subset I$.

\begin{lemma}\label{sr}
For all $j \in \N$, we have
\[
s_{j} x + \tfrac{3}{4} x - \tfrac{3}{8}(x_{-j}+x_j) =-3\al_{ \2}p_{\1,j}+3\al_{ \1}p_{\2,j}.
\]
\end{lemma}
\begin{proof}
By H2 and as $\sum_{i \in \Z} \al_i = 0$
, we get
\begin{align*}
s_{j} x &= -\tfrac{3}{4} x + \tfrac{3}{8}\sum_{i \in \Z} \al_i(a_{i-j}+a_{i+j}) - \sum_{i \in \Z} \al_i z_{\ii,j} \\
&=-\tfrac{3}{4} x + \tfrac{3}{8}(x_{-j}+x_j) - \sum_{\ii \in \Z_3} \al_{\ii} z_{\ii,j}.
\end{align*}
Now, by the definition of $z_{\ii,j}$ and since $\al_{\0}+\al_{\1}+\al_{\2}=0$, we get $\sum_{\ii \in \Z_3} \al_{\ii} z_{\ii,j}=3\al_{\2}p_{\1,j}-3\al_{\1}p_{\2,j}$.
\end{proof}

\begin{corollary}\label{Jinideal}
If either $\al_{\1}$, or $\al_{\2}$ is non-zero, then $J \subset I$.
\end{corollary}
\begin{proof}
By~Lemma \ref{sr}, as $\ch(\FF)\neq 3$, $\al_{\2}p_{\1,j}-\al_{\1}p_{\2,j}$ is a non-zero element of $I$ for all $j\in 3\N$.  Then, Lemma~\ref{Jminimal} implies $J \subset I$.
\end{proof}

So all the elements listed in Theorem \ref{idealspan} are contained in $I$.  We now show that these elements span a subspace $Y$ which is closed under multiplication by $\hatH$ and hence $I$ is indeed equal to $Y$.

\begin{proof}[Proof of Theorem $\ref{idealspan}$]
First note that the subspace $Y$ is closed under the action of $\Aut(\hatH)$ since its generating set is.  Secondly, as $\al_i = \epsilon \alpha_{D-i}$ for all $i \in \Z$, we have $y_k = \epsilon y_{D-k}$ and $p_k(\r) = \epsilon p_{D-k}(\r)$, and hence these are in $Y$ for all $k \in \Z$.

We begin by considering the products with the elements $x_k$. Since $a_ix_k=a_ix^{\theta_k}=(a_{i-k}x)^{\theta_k}$, $s_{j}x_k=(s_jx)^{\theta_k}$, and $p_{\r,j}x_k=(p_{\r-\kk, j}x)^{\theta_k}$ by the $\Aut(\hatH)$-invariance of $Y$ we just need to consider the products with $x$.  By Lemma~\ref{folding}, $a_j x =\tfrac{1}{2}x+y_j+p_j( \jj) \in Y$ for all $j \in \Z$.
By Lemma~\ref{sr} and Corollary~\ref{Jinideal}, $s_{j} x \in Y$ for all $j \in \N$. For $p_{\r,j}$, by H3, $p_{\r,j}x=-\sum_{i\in \Z} \al_i p_{-(\ii+\r),j}  = -\sum_{\ii \in \Z_3} \al_{\ii} p_{-(\ii+\r),j}$ which is zero if $\al_{\0} = \al_{\1} = \al_{\2} = 0$. Hence, in all cases,  $p_{\r,j}x\in Y$.

We now consider the products with $y_k= \sum_{i \in \N}(\al_{k-i} + \al_{k+i}) s_{i}$. For products $a_i y_k$, again using the $\Aut(\hatH)$-invariance of $Y$, it suffices to just consider $a_0 y_k$.
By H2, we have
\begin{align*}
a_0 y_k &= -\tfrac{3}{4} \sum_{i \in \N} (\al_{k-i} + \al_{k+i}) a_0 + \tfrac{3}{8}\sum_{i \in \N} (\al_{k-i} + \al_{k+i}) (a_{-i} + a_i) \\
&\phantom{{}={}} + \tfrac{3}{2} y_k - \sum_{i \in \N} (\al_{k-i} + \al_{k+i}) z_{\0,i}.
\end{align*}
As $\sum_{i \in \Z} \al_i = 0 $, we have $\sum_{i \in \N} (\al_{k-i} + \al_{k+i}) = \sum_{j \neq k} \al_j = -\al_k$. 
So, for the $a$-part of the above, we have
\begin{align*}
2\al_k a_0 + \sum_{i \in \N} (\al_{k-i} &+ \al_{k+i}) (a_{-i} + a_i)\\
&= 2\al_k a_0 + \sum_{i <0} \al_{k+i} (a_i + a_{-i})+ \sum_{i >0} \al_{k+i} (a_{-i} + a_i)\\
&= \sum_{i \in \Z} \al_{k+i}(a_{-i} + a_i)\\
&= \sum_{j \in \Z} \al_j (a_{k-j} + a_{-k+j}).
\end{align*}
Noting that ${ x_{-k}}^{\tau_0}= \sum_{i \in \Z} \al_i a_{k-i}$, we obtain
\[
a_0 y_k = \tfrac{3}{8} ({x_{-k}}^{\tau_0}+ x_{-k}) + \tfrac{3}{2} y_k -\left (p_k(\1) - p_k(\2)\right)
\]
which is in $Y$.  
For the products $s_{j} y_k$, we have
\begin{align*}
\tfrac{8}{3} s_{j} y_k &= \tfrac{8}{3} s_{j} \sum_{i \in \N} (\al_{k-i}+\al_{k+i}) s_{i} \\
&= \sum_{i \in \N }(\al_{k-i}+\al_{k+i})(2s_{j} + 2s_{i} - s_{|j-i|} - s_{j+i} )\\
&=-2\al_k s_{j}+2y_k - \sum_{i \in \N }(\al_{k-i}+\al_{k+i})s_{|j-i|} - \sum_{i \in \N }(\al_{k-i}+\al_{k+i})s_{j+i} 
\end{align*}
where we again use that $\sum_{i \in \N} (\al_{k-i}+\al_{k+i}) = -\al_k$.
Now we rewrite the last two sums by taking $l$ to be $|j-i|$ and $j+i$
\begin{align*}
\tfrac{8}{3} s_{j} y_k &= -2\al_k s_{j}+2y_k   - \sum_{l = 1}^{j-1}(\al_{k-j+l}+\al_{k+j-l})s_{l} \\
&\phantom{{}={}} - \sum_{l \in \N}(\al_{k-j-l}+\al_{k+j+l})s_{l} - \sum_{l >j}(\al_{k+j-l}+\al_{k-j+l})s_{l} \\
&= 2y_k  - \sum_{l \in \N}(\al_{k-j-l}+\al_{k-j+l} +\al_{k+j-l}+\al_{k+j+l})s_{l}\\
&= 2y_k - y_{k-j}-y_{k+j}
\end{align*}
which is in $Y$.
Replacing $s_j$ by $p_{\r,j}$, the same argument proves that $p_{\r,j}y_k\in Y$ for all $j\in \N$, $r\in \{1,2\}$.  

We are left with the products with $p_k(\r)= \sum_{i \in \N}(\al_{k-i} + \al_{k+i}) p_{\r,i}$. For the products $a_ip_k(\r)$, as above, by the $\Aut(\hatH)$-invariance of $Y$, it suffices to consider $a_0p_k(\r)$. By H3, $a_0p_k(\r)= \frac{3}{2} p_k(\r) - \frac{1}{2} p_k(-(\ii+\r)) \in Y$.  Since by H5, $s_jp_{\r,i}=s_ip_{\r,j}$, it follows that $s_jp_k(\r)=p_{r,k}y_j\in Y$. Finally, setting $\s = -(\r+\t)$, an analogous argument to that for $s_j y_k$ and $p_{\r,j} y_k$ shows that
\begin{align*}
8 p_{\t,j} p_k(\r) &= 2\left (p_k(\s+\1)-p_k(\s-\1)\right ) - \left (p_{k-j}(\s+\1)-p_{k-j}(\s-\1)\right )\\
&\phantom{{}={}} -\left (p_{k+j}(\s+\1)-p_{k+j}(\s+\1)\right )\in Y.
\end{align*}
Finally, as $l_s(y_k) = D-k = l_p(p_k(\r))$, it is clear that the elements given form a basis unless one of the elements is zero.  This can only happen if $D$ is even, $\epsilon = -1$ and $k=D/2$.
\end{proof}

%%%%%%%%%%%%%%%%%%%%%%%%%%%%%%%%%%%

\section{Two families of quotients}\label{sec:twofamilies}

In this section, we detail two families of ideals and their quotients in $\hatH$.

Firstly, suppose that $\hatH/I$ is a quotient with finitely many axes.  If $\hatH/I$ has $n$ axes, then $a_0-a_n \in I$.  In particular, $I_n := (a_0-a_n)$ is the minimal ideal such that the quotient has $n$ axes.

\begin{corollary}\label{In}
\ 
\begin{enumerate}
\item If $3 \nmid n$, then $J \subset I_n$ and $I_n$ has a basis given by
\begin{align*}
& a_i - a_{i+n} & \mbox{for } i \in \Z \\
& s_{j} - s_{j+n}, \  s_{jn} & \mbox{for } j \in \N \\
& s_{j} - s_{n-j} & \mbox{for } 1 \leq j \leq \left\lfloor \tfrac{n}{2} \right\rfloor
\end{align*}
and a basis for $J$.
\item If $3|n$, then $I_n$ has basis given by the above elements and
\begin{align*}
& p_{\r, 3j} - p_{\r,3j+n}, \  p_{\r, jn} & \mbox{for } j \in \N, r=1,2\\
& p_{\r,3j} - p_{\r,n-3j} & \mbox{for } 1 \leq j \leq \left\lfloor \tfrac{n}{2} \right\rfloor,  r=1,2
\end{align*}
\end{enumerate}
\end{corollary}
\begin{proof}
The ideal $I_n$ has {\pattern } $(\al_0, \dots, \al_n) = (1, 0, \dots, 0,-1)$.  So $\al_{\1} = \al_{\2} =0$ if and only if $3|n$.  By Theorem \ref{idealspan}, $J \subset I$ if $3  \nmid n$. In both cases we have basis elements $x_k = a_k - a_{k+n}$ for $k \in \Z$ and $y_k$.  
If $k < 0$, then $y_k = s_{|k|} - s_{n-|k|}$, if $k=0$, then $y_0 = s_{n}$, and if $1 \leq k \leq \left\lfloor \frac{n}{2} \right\rfloor$, then $y_k = s_{k} - s_{n-k}$. Similarly, in the case where $3 | n$, we get the corresponding expressions for the $p_k(\r)$'s.
\end{proof}

Define $\hatH_n:=\hatH/I_n$ and $\cH_n := \hatH/(J+I_n)$. Then $ \cH_n$ is isomorphic to a quotient of $\hatH_n$ and, since $\hatH/J\cong \cH$, it is also isomorphic to a quotient of $\cH$.

\begin{corollary}\label{quotientHn}
For every $n\in \N$, $\cH_n$ is a primitive $2$-generated axial algebra of Monster type $(2,\frac{1}{2})$ of dimension $n+ \left\lfloor \tfrac{n}{2}\right\rfloor$. If additionally $3 | n$ and $\ch(\FF)=5$, then $\hatH_n$ is a primitive $2$-generated axial algebra of Monster type $(2,\tfrac{1}{2})$ of dimension $n+ \left\lfloor \frac{n}{2}\right\rfloor+2 \left\lfloor \tfrac{n}{6} \right\rfloor$.
\end{corollary}

Note that $I_n$ is generated by a $-1$-type element $x$.  We now give an example of an ideal of $1$-type.  Let $L_n$ be the ideal generated by $2a_0 - (a_{-n} + a_n)$.

\begin{corollary}\label{Ln}
\ 
\begin{enumerate}
\item If $3  \nmid  n$, then $J \subset L_n$ and $L_n$ has a basis given by
\begin{align*}
& 2a_i - (a_{i-n}+a_{i+n}) & \mbox{for } i \in \Z \\
& s_{j} - 2s_{j+n} + s_{j+2n}, \  s_{jn} &\mbox{for } j\in \N \\
& s_{j} - 2s_{n-j} + s_{2n-j}  &\mbox{for } 1 \leq j < n
\end{align*}
and a basis for $J$.
\item If $3 | n$, then $L_n$ has basis given by the above elements and 
\begin{align*}
& p_{\r,j} - 2p_{\r,3j+n} + p_{\r,3j+2n}, \ p_{\r, jn} &\mbox{for } j\in \N, r=1,2\\
& p_{\r,3j} - 2p_{\r,n-3j} + p_{\r,2n-3j} &\mbox{for } 1 \leq j < n, r=1,2
\end{align*}
\end{enumerate}
\end{corollary}
\begin{proof}
The proof is obtained using Theorem \ref{idealspan} in an analogous way to Corollary \ref{In}.
\end{proof}

We also set $\hatL_n:=\hatH/L_n$ and $\cL_n:=\hatH/(J+L_n)$. Similarly to the previous case, $ \cL_n$ is isomorphic to a quotient of $\hatL_n$ and also to a quotient of $\cH$.

\begin{corollary}\label{quotientLn}
For every $n\in \N$, $\cL_n$ is a $2$-generated primitive axial algebra of Monster type $(2,\frac{1}{2})$ of dimension  $3n-1$. If additionally $3 | n$ and $\ch(\FF)=5$, then $\hatL_n$ is a primitive $2$-generated axial algebra of Monster type $(2,\frac{1}{2})$ of dimension $3n-1+2 \left\lfloor \tfrac{n-1}{3}\right\rfloor$. 
\end{corollary}

Note that, according to the characteristic of $\FF$, $\cL_n$ and $\hatL_n$ may have finitely or infinitely many axes.  In fact, if $\ch(\FF)=0$, they both have infinitely many axes, whereas if $\ch(\FF) = p$, one can show that they both have $pn$ axes.

%%%%%%%%%%%%%%%%%%%%%%%%%%%%%%%%%%

\section{ Exceptional isomorphisms}\label{sec:yabe} 

In \cite{yabe}, Yabe classifies symmetric $2$-generated primitive axial algebras of Monster type $(\al,\bt)$ in characteristic not $5$ (the characteristic $5$ case was completed by Franchi and Mainardis in \cite{highwater5}).

\begin{theorem}\textup{\cite{yabe, highwater5}}
A symmetric $2$-generated primitive axial algebra of Monster type $(\al,\bt)$ is isomorphic to one of the following:
\begin{enumerate}
\item a {$2$-generated primitive} axial algebra of Jordan type $\al$, or $\bt$;
\item a quotient of $\cH$, or $\hatH$ in characteristic $5$;
\item one of the algebras listed in \textup{\cite[Table $2$]{yabe}}.
\end{enumerate}
\end{theorem}

We wish to know which quotients of $\cH$, or $\hatH$, are actually isomorphic to one of the algebras in cases $1$, or $3$ above.  Clearly, we must have $(\al,\bt) = (2,\frac{1}{2})$. A direct check of~\cite[Table $2$]{yabe} gives the following list of the symmetric $2$-generated  primitive axial algebras of Monster type $(2,\frac{1}{2})$ (we use the notation from \cite{forbidden}):

\begin{enumerate}
\item $3\C(2)$;
\item one of the Jordan algebras (of Jordan type $\frac{1}{2}$) $S(\dl)$, for $\dl \neq 2$, $S(2)^\circ$, or $\widehat{S}(2)^\circ$;\footnote{These algebras were written $\Cl^J(\FF^2,b)$, $\Cl^0(\FF^2,b)$ and $\Cl^{00}(\FF^2,b)$, respectively, in \cite{Axial2}.  Note also that $3\C(\frac{1}{2}) \cong S(-1)$.} 
\item $\IY_3(2, \frac{1}{2}, \mu)$, for $\mu \in \FF$, and the quotient $\IY_3(2, \frac{1}{2}, 1)^\times$;\footnote{These are the algebras $\mathrm{III}(2, \frac{1}{2},-2\mu-1)$ and the quotient $\mathrm{III}(2, \frac{1}{2},-3)^\times$ in \cite[Table $2$]{yabe}. Note also that $3\A(2, \frac{1}{2}) = \IY_3(2, \frac{1}{2}, -\frac{1}{2})$.}
\item $\IY_5(2, \frac{1}{2})$ and the quotient $\IY_5(2, \frac{1}{2})^\times$;\footnote{These are $\mathrm{V}_2(2, \frac{1}{2})$ and $\mathrm{V}_2(2, \frac{1}{2})^\times$ in \cite[Table $2$]{yabe}. Note that the algebra $\mathrm{V}_1(2, \frac{1}{2})$ is defined in characteristic $5$ and it coincides with $\mathrm{V}_2(2, \frac{1}{2})$.}
\item in characteristic $7$, $4\A(2, \frac{1}{2})$ and its quotient $4\A(2, \frac{1}{2})^\times$;
\item in characteristic $5$, $6\A(2, \frac{1}{2})$.
\end{enumerate}

Note that, since every ideal $I$ of $\hatH$ is contained in the radical which is the kernel of the map $\lm$, we have an induced weight function $\bar{\lm}$ on $\hatH/I$ and $\hatH/I$ must be baric.  So, the only algebras in the list above which can be isomorphic to a quotient of $\hatH$ are ones which are also baric.  By \cite[Proposition 5.5]{forbidden}, $S(\dl)$ is simple if $\dl \neq \pm 2$, so it cannot be baric.  Also by~\cite[Proposition 5.5]{forbidden}, $S(-2)$ has precisely two codimension $1$ ideals, but in both cases, one of the generators is contained in a codimension $1$ ideal.  Since in a quotient of $\hatH$ neither generator is contained in the kernel of $\bar{\lm}$, $S(-2)$ cannot be isomorphic to any quotient of $\hatH$.  The remaining possibilities and their quotients do indeed all occur as quotients of $\hatH$.

\begin{theorem}\label{Yabeiso}
The algebras $3\C(2)$, $S(2)^\circ$, $\widehat{S}(2)^\circ$, $\IY_3(2, \frac{1}{2}, \mu)$, for $\mu \in \FF$, $\IY_5(2, \frac{1}{2})$ and $6\A(2, \frac{1}{2})$ in characteristic $5$, \textup{(}and their quotients\textup{)} are all quotients of $\hatH$.
\end{theorem}

In fact we will see below that all the algebras except $6\A(2, \frac{1}{2})$ are quotients of the Highwater algebra $\cH$.  The algebra $6\A(2, \frac{1}{2})$ in characteristic $5$ is not a quotient of the Highwater algebra $\cH$, but is a quotient of the cover $\hatH$.

We will prove this theorem via a series of lemmas.  Since the algebras in the statement are finite dimensional, their axes satisfy a non-trivial linear relation.  For each algebra $A$, we will exhibit an element $x \in \hatH$ so that $\hatH/(x) \cong A$.  In fact, in all but two cases, $x$ has trivial $s$-part. For $v\in \hatH$, we will write  $\bar v$ for the image of $v$ in $\hatH/I$. For the proof we will require some details about each of the target algebras.  We do not give those here, but they can be found in \cite{forbidden} and \cite{yabe}.

\begin{lemma}\label{H2 L2}
We have $\hatH_2=\cH_2 \cong 3\C(2)$ and $\hatL_1 = \cL_1 \cong S(2)^\circ$.
\end{lemma}
\begin{proof}
By Corollary~\ref{In}, $\hatH_2=\cH_2$ has basis $\qa_0, \qa_1, \qs_{1}$.
 Since $a_{-i}-a_i\in I_2$ for all $i \in \N$, the $\frac{1}{2}$-eigenspace for $\ad_{\qa_0}$ is trivial and so $\hatH_2$ is a primitive $2$-generated axial algebra of Jordan type $2$. Hence, by~\cite[Theorem 1.1]{Axial2}, it is isomorphic to $3\C(2)$.

Similarly, by Corollary \ref{quotientLn}, $\hatL_1 = \cL_1 = \la \qa_0, \qa_1 \ra$ is $2$-dimensional.  Since $ a_{-1}-a_{1} \notin L_1$, $\tau_0$ induces a non-trivial automorphism on $\cL_1$, so $\cL_1$ is a primitive $2$-dimensional axial algebra of Jordan type $\frac{1}{2}$ and therefore must be isomorphic to $S(2)^\circ$ by~\cite[Theorem 1.1]{Axial2}.
\end{proof}

We note that $(1,0,-1)$ and $(-1,2,-1)$ are the only two ideal tuples of length $3$ up to scaling (cf. Theorem \ref{classification I not in J}).

The following gives a positive answer to an open question in~\cite[Question 4.5]{splitspin}.

\begin{lemma}\label{caseIY_3}
Let $I_\dl := (a_0 + \dl a_1 - \dl a_2 -a_3)$ for $\dl \in \FF$.  Then the quotient $\hatH/I_{\dl}$ is isomorphic to $\IY_3(2, \frac{1}{2}, \mu)$, where $\dl = -2\mu -1$.
\end{lemma}
\begin{proof}
By Theorem \ref{idealspan}, $(1-\dl)s_{1} - s_{2}$, $s_{j} + \dl s_{j+1} - \dl s_{j+2} - s_{j+3}$ and $p_{\r,3j}$ are in $I_\dl$, for all $j \geq 0$ and $r = 1,2$.  Hence $\bar a_{-1}, \bar a_0, \bar a_1, \bar s_{1}$ is a basis for $A := \hatH/I_\dl$. (In particular, $J \subset I_\dl$, even if $\dl=0$).  Define $q := -\frac{3}{4}((\dl+1)a_0+a_{-1}+a_1) + s_{1}$ and hence $s_{1} = q+ \frac{3}{4}((\dl+1)a_0+a_{-1}+a_1)$.  We claim that $\bar a_{-1}, \bar a_0, \bar a_1, \bar q$ satisfy the same products as given by Yabe in \cite[Section 3.2]{yabe}.

It is immediate that $\bar a_i\bar a_{i+1} = \frac{1}{2}(\bar a_i+\bar a_{i+1}) + \bar s_{1} = \frac{1}{2}(\bar a_i+\bar a_{i+1}) + \bar q+ \frac{3}{4}((\dl+1)\bar a_0+\bar a_{-1}+\bar a_1)$, where $i = 0,-1$, and $\bar a_{-1}\bar a_1 = \frac{1}{2}(\bar a_{-1}+\bar a_1) + \bar s_{2} = \frac{1}{2}(\bar a_{-1}+\bar a_1) + (1-\dl)\bar s_{1} = \frac{1}{2}(\bar a_{-1}+\bar a_1)+ (1-\dl)(\bar q+ \frac{3}{4}((\dl+1)\bar a_0+\bar a_{-1}+\bar a_1))$ as required.  It is a straightforward, but somewhat long calculation to show that $\bar q\bar x = \frac{3}{4}(\dl+3)\bar x$ for all $\bar x \in A$.
\end{proof}

Recall from Section \ref{sec:fusionlaw}, that $v_1 = 2a_0 - (a_{-1}+a_1) - 4s_1 \in A_2(a_0)$.

\begin{corollary}\label{hatS2circquo}
We have $\hatH/(v_1) \cong \widehat{S}(2)^\circ$.
\end{corollary}
\begin{proof}
Observe that $v_1^{\flip} = 2a_1-(a_2+a_0) -4s_{1}$ and so $v_1^{ \flip} - v_1 = a_{-1} -3a_0+3a_1-a_2$.  Hence $(v_1) \leq (a_0 -3 a_1 +3 a_2 -a_3)= I_{-3}$ and so, by Lemma \ref{caseIY_3}, $\hatH/(v_1)$ is a quotient of $B :=\IY_3(2, \frac{1}{2}, 1)$.  Note that the image $\bar v_1$ of $v_1$ in $B$ is non-trivial and so it is a $2$-eigenvector for $\qa_0$ in $B$.  Since the {eigenvalues of} $\qa_0$ in the $4$-dimensional $B$ are known to be $1$, $0$, $2$, $\frac{1}{2}$, $B_2(\qa_0)$ is $1$-dimensional and hence is spanned by $\bar{v}_1$.
 Since $v_1^{ \flip} - v_1 \in I_{-3}$, we have $\bar v_1 = \bar v_1^\flip$ in $B$.  However, $\la \bar v_1^\flip \ra = \la \bar v_1 \ra ^\flip = B_2(\qa_0)^\flip = B_2(\qa_1)$.  Therefore, as $B$ is generated by $\qa_0$ and $\qa_1$, $(\bar v_1)$ is  a $1$-dimensional ideal in $B$.  Therefore, $\hatH/(v_1) \cong B/(\bar{v}_1)$ is a $3$-dimensional primitive axial algebra of Jordan type $\frac{1}{2}$.  From our list, the only possibility of dimension $3$ is $\widehat{S}(2)^\circ$ (and $\qs_{1}$ is the nilpotent element).
\end{proof}

From~\cite[Section~3.6]{yabe}, $\IY_5(2, \frac{1}{2})$ has basis $(\hat a_{-2}, \hat a_{-1}, \hat a_{0}, \hat a_{1}, \hat a_{2}, \hat p_{1})$ and the axes satisfy the relation $\hat a_{-2}-5\hat a_{-1}+10\hat a_0-10\hat a_1+5\hat a_2-\hat a_3 = 0$.

\begin{lemma}\label{caseIY_5}
Let $y :=  a_{-2} -4a_{-1} + 6a_0 -4a_1+a_2 $ and $y_1:=y -16s_{1}+4s_{2}$.  Then $(y_1)\subseteq (y)$, $\hatH/(y_1) \cong \IY_5(2,\frac{1}{2})$ and $\hatH/(y) \cong \IY_5(2,\frac{1}{2})^\times$.
\end{lemma}
\begin{proof}
We just sketch the proof as it is similar to those above. Note that, by Theorem~\ref{idealspan}, $(y)$ contains the element $-8s_{1}+2s_{2}$ and hence $y_1$. Moreover  $x:=y_1-y_1^{\flip}=a_{-2}-5a_{-1}+10 a_0-10 a_1+5 a_2- a_3 \in (y_1)$.
By Theorem~\ref{idealspan}, $\hatH/(x)$ is $7$-dimensional with basis given by $\qa_{-2}, \dots, \qa_2, \qs_{1}, \qs_{2}$.  One can now check that $\bar y_1v \in  \la \bar y_1 \ra$  in $\hatH/(x)$ for every $v\in \{\qa_{-2}, \dots, \qa_2, \qs_{1}, \qs_{2}\}$ and so $\hatH/(y_1)$ is $6$-dimensional.  Another calculation shows that the linear map from $\hatH/(y_1)$ to $\IY_5(2, \frac{1}{2})$, sending $\qa_i$ to $\hat{a}_i$ and $\qs_{1}$ to  $\hat{p}_1$ is an isomorphism. Finally, by~\cite[Section~3.6]{yabe}, $ \IY_5(2,\frac{1}{2})^\times=\IY_5(2, \frac{1}{2})/(\hat a_{-2} -4\hat a_{-1} + 6\hat a_0 -4\hat a_1+\hat a_2)$ and the result follows.
\end{proof}

We now consider the two exceptional cases in finite characteristic.

\begin{lemma}
Let $\FF$ have characteristic $7$.  The algebras $4\A(2, \frac{1}{2})$ and $4\A(2, \frac{1}{2})^\times$ are not quotients of $\hatH$.
\end{lemma}
\begin{proof}
Suppose for a contradiction that $4\A(2, \frac{1}{4})$, or $4\A(2, \frac{1}{2})^\times$, is a quotient of $\hatH$.  In both cases, the algebra has a closed set of four axes and so must be isomorphic to some quotient $\hatH_4/I$ of $\hatH_4$.  Also in both algebras, $\lla a_i, a_{i+2} \rra \cong 2\B$.  So in $\hatH_4$, $\qa_0 \qa_2 = \frac{1}{2}(\qa_0 + \qa_2) + \qs_{2} \in I$.  Now
\begin{align*}
\qa_0(\qa_0 \qa_2) &= \tfrac{1}{2} \qa_0 + \tfrac{1}{2}(\tfrac{1}{2}(\qa_0+\qa_2) + \qs_{2}) -\tfrac{3}{4} \qa_0 + \tfrac{3}{8}(\qa_{-2} + \qa_2) + \tfrac{3}{2}\qs_{2} \\
&= \qa_2 + 2\qs_{2}
\end{align*}
is in $I$ and hence $\qa_0 = 2\qa_0\qa_2 - \qa_0(\qa_0 \qa_2) \in I$.  By Theorem \ref{symmetry}, $\qa_1 \in I$ and so $I = \hatH_4$, a contradiction.
\end{proof}

Finally, we consider $6\A(2, \frac{1}{2})$ in characteristic $5$ (see also~\cite[p.\ 208]{highwater5}).  In~\cite[Table 1]{forbidden}, we see that $6\A(2, \frac{1}{2})$ has basis $\hat a_0, \ldots , \hat a_5, \hat c, \hat z$ and $\lla \hat a_i, \hat a_{i+3}\rra \cong 3\C(2)$ with the third axis equal to $\hat c$, for every $i$ modulo $6$.

\begin{lemma}
Let $\FF$ have characteristic $5$.  Then, $\hatH/(a_0-a_1+a_3-a_4+p_{\2,3}) \cong 6\A(2, \frac{1}{2})$.
\end{lemma}

\begin{proof}
Set $x:=a_0-a_1+a_3-a_4+ p_{\2,3}$. By Corollary~\ref{In}, $\hatH_6$ is $11$-dimensional with basis $\qa_0, \dots, \qa_5$, $\qs_{1}$, $\qs_{2}$, $\qs_{3}$, $\bar p_{\1,3}$ and $\bar p_{\2,3}$.  One can check that the linear map from $\hatH_6$ to $6\A(2, \frac{1}{2})$ defined by
\begin{gather*}
\qa_i\mapsto\hat a_i \mbox{ for } i\in \{0,\ldots , 5\}, \\
\qs_{1} \mapsto \sum_{i=0}^5\hat a_i-\hat c-\hat z, \quad
\qs_{2}\mapsto  \hat z,\quad
\qs_{3}\mapsto  \sum_{i=0}^5\hat a_i-\hat c, \\
\bar p_{\1,3} \mapsto \hat a_0 -\hat a_2+\hat a_{3}-\hat a_5,\quad
\bar p_{\2,3} \mapsto  -\hat a_0 +\hat a_1-\hat a_{3}+\hat a_4,
\end{gather*}
is a surjective algebra homomorphism and  $\bar x$ belongs to the kernel $I$ of this map. Hence $\bar x^2= \qs_{1}+  \qs_{2} -\qs_{3} \in I$ and, by Theorem~\ref{symmetry}, also $x^{\theta_1}\in I$. Since $\bar x$, $\bar x^2$, and $\bar x^{\theta_1}$  are linearly independent, $(\bar x)$ has dimension at least $3$. As $I$ has dimension $3$, we get $I=(\bar x)$, whence $\hatH/(a_0-a_6, x)\cong 6A(2, \frac{1}{2})$.  Finally, note that $a_0-a_6=x-x^{\tau_3}+x^{\theta_1}$ and hence $(a_0-a_6, x)=(x)$.
\end{proof}

This completes the proof of Theorem \ref{Yabeiso}.
\medskip

There are however other possibilities for isomorphisms with algebras on Yabe's list, namely if the quotient is isomorphic to an $\cM(\frac{1}{2}, 2)$-axial algebra.  Such a quotient $A$ of $\hatH$ would have a fusion law which admits a $C_2$-grading with respect to both $\frac{1}{2}$ and $2$.  

\begin{theorem}\label{2graded exceptional isos}
Let $A$ be a non-trivial quotient of $\hatH$ with fusion law $\cF_A$ naturally induced from $\cF$.  Suppose that $\gr \colon \cF_A \to T$ is a finest adequate grading of $\cF_A$ such that $\gr(2) \neq 1_T$.  Then $A$ is isomorphic to a quotient of one of
\begin{enumerate}
\item $\hatH_2 \cong 3\C(2)$,
\item $\hatH/(a_{-1} -a_{0} -a_1 +a_2+2s_{2}) \cong 6\Y(\frac{1}{2}, 2)$,
\item  $\hatH/I_{-3} \cong \IY_3(2, \frac{1}{2}, 1)$, or
\item $\hatH/(3(a_{-1} -a_{0} -a_1 +a_2)-2s_{2})$, a $5$-dimensional algebra.
\end{enumerate}
In particular, $A$ is a quotient of the Highwater algebra $\cH$.
\end{theorem}
\begin{proof}
Let $A = \lla \qa_0, \qa_1 \rra=\hatH/I$. By Corollary~\ref{corsymmetry}, $A$ is symmetric and hence $\ad_{\qa_0}$ and $\ad_{\qa_1}$ have the same eigenvalues. Since $\cF_A$ is induced from $\cF$, we have $\cF_A \subseteq \{ 1,\frac{5}{2}, 0, 2, \frac{1}{2}\}$, with $\lm \star_{\cF_A} \mu \subseteq \lm \star_{\cF} \mu$, but this containment may be proper in some cases. By hypothesis, $2\in \cF_A$ and so $A_{2} \neq 0$. In particular, $A$ is not a quotient of $ \widehat{S}(2)^\circ\cong \hatH/(v_1)$, where $\qa_0$ has trivial $2$-eigenspace. Hence $v_1 \notin I$.

If $\frac{1}{2} \notin \cF_A$, then  $A_{\frac{1}{2}}(\qa_0) = 0$.  In particular, $w_1 \in I$ and so $A$ is a quotient of $\hatH/(a_{-1}-a_1) = \hatH_2 \cong 3\C(2)$, by Lemma \ref{H2 L2}.  Hence from now on, we assume that $\frac{1}{2} \in \cF_A$ and $\bar{w}_1  \neq 0$.  By Lemma~\ref{grading}, $\cF$ has a $C_2$-grading where $\frac{1}{2}$ is graded non-trivially and so $\cF_A$ also has a grading where $\frac{1}{2}$ is graded non-trivially.
Moreover, as every grading factors through the finest grading, $g_{\sfrac{1}{2}} := \gr(\frac{1}{2}) \neq 1_T$ (however $g_{\sfrac{1}{2}}$ may have infinite order, or any order divisible by $2$).

Observe that if $\lm \in \lm \star_{\cF_A} \mu$ for $\lm, \mu \in \cF$, then $\gr(\lm) = \gr(\lm)\gr(\mu)$ and so $\gr(\mu) = 1$.  In particular, if $\frac{1}{2} \in 2 \star_{\cF_A} \frac{1}{2}$, then $g_2 := \gr(2) = 1_T$, a contradiction.  Hence $2 \star_{\cF_A} \frac{1}{2} = \emptyset$ and so $\bar{w}_1\bar v_1 =0$ in $A$. Since
\begin{align*}
w_1 v_1 &= (a_{-1} - a_1)(2a_0 - (a_{-1} + a_1) - 4s_{1}) \\
&= (a_{-1} - a_1) - (a_{-1} - a_1) -4\left( -\tfrac{3}{4}(a_{-1} - a_1) + \tfrac{3}{8}(a_{-2} + a_0 - a_0  - a_2) \right)\\
&= -\tfrac{3}{2}\left( a_{-2} -2a_{-1} + 2a_1 -a_2 \right),
\end{align*} 
we must have $r := a_{-2} -2a_{-1} + 2a_1 -a_2 \in I$.  By Theorem~\ref{idealspan}, $J \subset (r)\subseteq I$, and so $A_{\sfrac{5}{2}}(\bar a_0) = \emptyset$. Moreover, $-s_{1} - 2s_{2} + s_{3}$ and $s_j-2s_{j+1} + 2s_{j+3} - s_{j+4}$ are in $(r) \subseteq I$ for all $j\geq 0$.  Hence, $\hatH/(r)$ is $6$-dimensional with basis given by the images of $a_{-1}, a_0, a_1, a_2, s_{1}, s_{2}$.  We will work inside $\hatH/(r)$: for $v\in \hatH$, we denote by $\tilde v$ its image in $\hatH/(r)$. Note that $\tilde u_2$ and $\tilde v_2$ are both non-zero.

We split now into two cases: either $2 \in \frac{1}{2} \star_{\cF_A} \frac{1}{2}$, or not.  We have
\begin{align}\label{w2}
w_1^2 &= (a_{-1} - a_1)^2 \nonumber\\
&= a_{-1} + a_1 - (a_{-1} + a_1) - 2s_2 \\
&=  -2s_2 = -\tfrac{1}{8}(u_2 - 3v_2).\nonumber
\end{align}

Assume first that $2 \notin \frac{1}{2} \star _{\cF_A} \frac{1}{2}$. Then $\bar w_1^2\in A_{0}(\qa_0)$ and hence Equation~(\ref{w2}) implies $v_2 \in I$.  We claim that $\tilde{v}_2^{\flip} = \tilde{v}_2$.  We have $\tilde{c}_2^{\flip} = (2\tilde a_0 - (\tilde a_{-2} + \tilde a_2))^{\flip} = 2\tilde a_1 - (\tilde a_3+\tilde a_{-1}) =  2\tilde a_1 - (\tilde a_{-1} - 2\tilde a_0 + 2\tilde a_2 +\tilde a_{-1}) = 2\tilde a_0 - ((2\tilde a_{-1} - 2\tilde a_1 + \tilde a_2) +\tilde a_2) = 2\tilde a_0 - (\tilde a_{-2} +\tilde a_2) = \tilde{c}_2$.  Hence $\tilde{v}_2^{\flip} = \tilde{v}_2$ in $\hatH/(r)$ as claimed.
Since $\tilde{v}_2$ is invariant under $\tau_0$ and $\flip$, it is fixed by the action of every automorphism induced by $\Aut(\hatH)$ on $\hatH/(r)$.   Hence $\tilde {v}_2$ is a common $2$-eigenvector for all axes $\tilde a_i$. As $\hatH/(r)$ is generated by $\tilde a_0$ and $\tilde a_1$, $(\tilde{v}_2)$ is a $1$-dimensional ideal.  So $\hatH/(r, v_2)$ is $5$-dimensional.  Note that $-\frac{1}{2}(v_2+r) = a_{-1} -a_0 -a_1 +a_2+2s_{2} =: x$.  Conversely, $x^{\theta_{-1}} - x = r$ and $-x^{\theta_{-1}} - x = v_2$ and so $(x) = (r, v_2)$.

We claim that $\hatH/(x) \cong 6\Y(\frac{1}{2},2)$.  Let $\bar{y}$ now denote the image of $y \in \hatH$ in $\hatH/(x)$.  Set $b_0:= \qa_0$, $b_2 := \frac{1}{4}(3\qa_0 + 2\qa_1 -\qa_2 - \qs_1)$, $b_4 := \frac{1}{4}(\qa_{-1} +\qa_0 + 3\qa_1 -\qa_2)$, $d := a_1 - b_4 = \frac{1}{4}(-\qa_{-1} -\qa_0 + \qa_1 +\qa_2)$ and $z := \frac{1}{4}\qs_2 = \frac{1}{8}(-\qa_{-1}+\qa_0+\qa_1-\qa_2)$.  A calculation shows that $(b_0, b_2, b_4, d, z)$ is a basis for $\hatH/(x)$ that satisfies the multiplication table for $6\Y(\frac{1}{2}, 2)$ as given in \cite{forbidden} (and $b_1 = \qa_1$).  One can check that the fusion law here is $C_2 \times C_2$-graded with generators $g_{\sfrac{1}{2}}$ and $g_2$.

Assume now that $2 \in \frac{1}{2} \star_{\cF_A} \frac{1}{2}$.  We have two subcases depending on whether $0 \in \frac{1}{2} \star_{\cF_A} \frac{1}{2}$.  If $0$ is in $\frac{1}{2} \star_{\cF_A} \frac{1}{2}$, then $g_0 := \gr(0) = g_2 \neq 1_T$.  So, as we observed above, $0\notin 0 \star_{\cF_A} 0$, whence $0 \star_{\cF_A} 0 = \emptyset$.  Hence, by Lemma~\ref{product uu}, we have $u_1^2 = 3(-4u_1 + u_2) \in I$.  Setting $y := -4u_1 + u_2$, we have $y - y^{\theta_1} = 3(-a_{-2} + 5a_{-1} - 10a_0 + 10a_1 - 5a_2 + a_3) \in I$ and so $\frac{1}{3}(y - y^{\theta_1}) + r + r^{\theta_1} = 4(a_{-1} - 3a_0 + 3a_1 -a_2) \in I$.  Let $x := a_{-1} - 3a_0 + 3a_1 -a_2$.  We claim that $(x) = (r, y)$.  Clearly $(x) \subseteq (r, y)$.  Conversely, we have $x+x^{\theta_{-1}} = r$.  By Theorem \ref{idealspan}, $-4s_1 + s_2 \in (x)$ and so $y =3( x-x^{\theta_{-1}} )-4(4s_1 - s_2) \in (x)$.  Hence $(x) = (r, y)$ as claimed.  By Lemma~\ref{caseIY_3}, $(x) = I_{-3}$, so $\hatH/(x) \cong \IY_3(2, \frac{1}{2}, 1)$ and $A$ is isomorphic to a quotient of $\IY_3(2, \frac{1}{2}, 1)$.  One can check that, apart from $\frac{1}{2} \star_{\cF_A}\frac{1}{2} = \{0,2\}$, we have $\lm \star_{\cF_A} \mu = \emptyset$ for all $\lm, \mu \neq 1$.  So the fusion law for $\IY_3(2, \frac{1}{2}, 1)$ is $\Z$-graded, with $\la g_{\sfrac{1}{2}} \ra \cong \Z$ and $g_0 = g_2 = g_{\sfrac{1}{2}}^2$.

Finally, if $0$ is not in $\frac{1}{2} \star_{\cF_A} \frac{1}{2}$, then $\bar w_1^2\in A_{2}(\qa_0)$ and hence, by Equation~(\ref{w2}), $u_2 \in I$. Recall from above that $\tilde{c}_2^{\flip} = \tilde{c}_2$ in $\hatH/(r)$.  Hence, similarly to $\tilde{v}_2$, $\tilde{u}_2^{\flip} = \tilde{u}_2$ and so $(\tilde{u}_2)$ is a $1$-dimensional ideal of $\hatH/(r)$.  Let $x := 3(a_{-1} -a_{0} -a_1 +a_2)-2s_{2} = -\frac{1}{2}(u_2+3r)$.  So clearly $(x) \subseteq (r,u_2)$.  Since $x-x^{\theta_{-1}} =  -3r$ and $u_2 = -2x-3r$, we have $(x) = (r,u_2)$ and hence $\hatH/(x)$ is $5$-dimensional.  When $A=\hatH/(x)$, since $\frac{1}{2} \star_{\cF_A} \frac{1}{2} = \{ 2 \}$, $2 \star_{\cF_A} 2 =\{ 0\}$, and $0 \star_{\cF_A} 0 =\{ 0\}$, we observe that the fusion law is $C_4$-graded, with $\la g_{\sfrac{1}{2}} \ra \cong C_4$ and $g_2 = g_{\sfrac{1}{2}}^2$.
\end{proof}
 
Note that the above algebras are graded by $C_2$, $C_2 \times C_2$, $\Z$ and $C_4$, respectively.  For $\hatH_2 \cong 3\C(2)$, the set of three axes $\qa_0$, $\qa_1$ and $\qa_0+\qa_1 - \qa_0\qa_1$ is closed under to the action of the Miyamoto group (with respect to the grading $C_2$ on the $2$-part).

The fusion law for $6\Y(\frac{1}{2}, 2)$ is $(C_2 \times C_2)$-graded.  Specifically, for an axis $a$, there are three distinct non-trivial Miyamoto involutions associated to $a$ (and belonging to the axis subgroup corresponding to $a$; see~\cite[Definition~3.2]{axialstructure}): the map $\tau_a(2)$ inverting the $2$-part and fixing the remaining eigenspaces, the map $\tau_a(\frac{1}{2})$ inverting the $\frac{1}{2}$-part and fixing the remaining eigenspaces and  the product of these two.  As $6\Y(\frac{1}{2}, 2) \cong \hatH/(a_{-1} -a_{0} -a_1 +a_2+2s_{2})$, we see that $X = \{ \qa_i : i \in \Z \}$ is generically an infinite set of axes closed under the action of the infinite dihedral group $\la \tau_0(\frac{1}{2}), \tau_1(\frac{1}{2}) \ra$ (both these can be finite in finite characteristic).  Hence taking only the $C_2$-grading with respect to the $\frac{1}{2}$-eigenspace, $6\Y(\frac{1}{2}, 2)$ is a 2-generated $\cM(2, \frac{1}{2})$ axial algebra with infinitely many axes.

However, taking just the $C_2$-grading with respect to the $2$-eigenspace, $6\Y(\frac{1}{2}, 2)$ is a $2$-generated $\cM(\frac{1}{2}, 2)$-axial algebra~\cite{yabe}.  Its Miyamoto group is $\la \tau_0(2), \tau_1(2) \ra \cong S_3$ and the closure under the Miyamoto group of the generating set $\{\qa_0, \qa_1\}$ has size $6$ \cite[Section 7.2]{forbidden} (the $6$ in the name $6\Y(\frac{1}{2},2)$ gives the number of axes in a closed set of generators).  Note that in this case, it does not appear that the closure of $\{\qa_0, \qa_1\}$ is a subset of $X$.  Taking the full $C_2 \times C_2$ grading we would get a much larger set of axes closed under the action of the Miyamoto group.

For the third and fourth cases above, we need to take a field with sufficiently many roots of unity in order to exhibit the full Miyamoto group.  Recall from \cite[Section~3]{axialstructure} that for each axis $a$ and character $\chi \in T^*$, we get a Miyamoto automorphism $\tau_a(\chi)$ defined by $v \mapsto \chi(t)v$ where $v$ is an eigenvector in a $t$-graded part.  The axis subgroup $T_a:= \la \tau_a(\chi) : \chi \in T^* \ra \cong T^*$ is isomorphic to a quotient group of $T$ depending on the field.  The fourth case $\hatH/(3(a_{-1} -a_{0} -a_1 +a_2)-2s_{2})$ is $C_4$-graded, so taking a field which contains $4$\textsuperscript{th} roots of unity, we get the axis subgroup $T_a \cong C_4$ and the Miyamoto group is as `large' as possible.  To exhibit the full Miyamoto group for the third case, $\IY_3(2, \frac{1}{2}, 1)$, in characteristic $0$, we need to work over $\mathbb{C}$.

\end{document}